\documentclass[12pt,leqno,a4paper]{amsart}
\usepackage{amssymb}
\usepackage[all]{xy}
\newcommand{\CC}{{\mathbb{C}}}
\newcommand{\ZZ}{{\mathbb{Z}}}

\newcommand\bs{{\mathbf s}}
\newcommand\bt{{\mathbf t}}
\newcommand\bu{{\mathbf u}}
\newcommand\bv{{\mathbf v}}
\newcommand\bw{{\mathbf w}}
\newcommand\bx{{\mathbf x}}
\newcommand\bz{{\mathbf z}}
\newcommand\bW{{\mathbf W}}
\newcommand\bS{{\mathbf S}}
\newcommand{\QQ}{{\mathbb Q}}

\newcommand{\cH}{{\mathcal{H}}}

\newcommand\bpi{{\boldsymbol\pi}}

\newcommand{\Irr}{{\operatorname{Irr}}}
\newcommand{\GL}{{\operatorname{GL}}}
\newcommand{\Frac}{{\operatorname{Frac}}}

\newcommand{\reg}{{\operatorname{reg}}}

\let\vhi=\varphi
\def\lexp#1#2{\kern\scriptspace\vphantom{#2}^{#1}\kern-\scriptspace#2}

\newtheorem{lem}[equation]{Lemma}
\newtheorem{conj}[equation]{Conjecture}

\newtheorem{prop}[equation]{Proposition}
\numberwithin{equation}{section}

\theoremstyle{definition}

\newtheorem{defn}[equation]{Definition}

\theoremstyle{remark}
\newtheorem{rem}[equation]{Remark}

\begin{document}

\title{Constructing representations of Hecke algebras for complex
       reflection groups}

\date{September 18, 2009}

\author{Gunter Malle}
\address{FB Mathematik, TU Kaiserslautern,
Postfach 3049, 67653 Kaisers\-lautern, Germany.}
\makeatletter
\email{malle@mathematik.uni-kl.de}
\makeatother

\author{Jean Michel}
\address{CNRS, Universit\'e Paris VII, 175, rue du Chevaleret,
75013 Paris, France.}
\makeatletter
\email{jmichel@math.jussieu.fr}
\makeatother

\thanks{The authors thank the Mathematical Sciences Research Institute,
Berkeley, and the Isaac Newton Institute for Mathematical Sciences, Cambridge,
for their hospitality during the preparation of parts of this work}

\begin{abstract}
We investigate the representations and the structure of Hecke algebras
associated to certain finite complex reflection groups. We first describe
computational methods for the construction of irreducible representations of
these algebras, including a generalization of the concept of $W$-graph to
the situation of complex reflection groups. We then use these techniques to
find models for all irreducible representations in the case of complex
reflection groups of dimension at most three. Using these models we are able
to verify some important conjectures on the structure of Hecke algebras.
\end{abstract}

\maketitle

\pagestyle{myheadings}
\markboth{for personal use only}{Constructing representations}

\section{Introduction} \label{sec:intro}

Let $W\le\GL(V)$ be a finite irreducible group on a complex vector space $V$
generated by complex reflections, that is, $W$ is a finite complex reflection
group. Let $R\subset W$ denote the set of reflections in $W$. For any
reflection $s\in R$ let $H_s\subset V$ denote its hyperplane of fixed points
on $V$. Then $V^\reg:=V\setminus\cup_{s\in R} H_s$ is connected in the complex
topology, and $W$ acts freely (and continuously) on $V^\reg$ by the theorem of
Steinberg. The {\em braid group} associated to $(W,V)$ is the fundamental group
$$B(W):=\pi_1(\bar V,\bar x_0)$$
of the quotient $\bar V:= V^\reg/W$ with respect to some base point
$\bar x_0\in\bar V$. \par
Let $H$ be the reflecting hyperplane of some reflection of $W$. Then its
stabilizer $W_H$ is cyclic, consisting solely of reflections (and the
identity). The \emph{distinguished reflection} $s_H\in W_H$ of $W_H$ is by
definition
the reflection whose non-trivial eigenvalue on $V$ equals $\exp(2\pi i/d)$,
where $d:=|W_H|$. Via the natural projection map from $B(W)$ onto $W$
induced by the quotient map $V^\reg\rightarrow\bar V$, the distinguished
reflection $s_H$ can be lifted to so-called \emph{braid reflections}
$\bs$ in $B(W)$.
For each reflection $s_H$ choose $d$ indeterminates $u_{s,0},\ldots,u_{s,d-1}$
such that $u_{s,j}=u_{t,j}$ if $s,t$ are conjugate in $W$. We write
$\bu$ for the collection of these indeterminates, and let
$A:=\ZZ[\bu,\bu^{-1}]$. The {\em generic cyclotomic Hecke algebra} associated
to $W$ with parameters $\bu$ is the quotient
$$\cH(W,\bu):=A\,B(W)/I$$
of the group algebra $A\,B(W)$ of the braid group $B(W)$ by the ideal $I$
generated by the
$\prod_{i=0}^{d-1}(\bs-u_{s,i})$, where $s$ runs over the distinguished
reflections and $\bs$ over the associated braid reflections. \par
An important and well-studied special case occurs if $W$ is actually a real
reflection group, that is a Coxeter group, in which case the cyclotomic
Hecke algebra becomes the well-known Iwahori-Hecke algebra of $W$. In this
situation, all of the questions mentioned below have been settled quite
a while ago, so here we will be concerned exclusively with the non-real
groups. \par
Bessis \cite[0.1(e)]{BZ} has shown that $B(W)$ has a presentation of the form
\begin{equation}\label{presentation}
  \langle \bs_1,\ldots,\bs_n\mid
  p_j(\bs_1,\ldots,\bs_n)=q_j(\bs_1,\ldots,\bs_n)\rangle
\end{equation}
where  $\bs_i$  are braid reflections whose images in $W$ form a minimal
system of reflections needed to generate $W$ (thus if $W$  is irreducible
we have $n=\dim V$ or $n=\dim  V+1$) and where $(p_j,q_j)$ run over a finite
set of pairs positive words of equal length in the $\bs_i$. One obtains a
presentation of $W$ by adding the relations
$\bs_i^{d_{s_i}}=1$  where $d_{s_i}$ is the order of the reflection $s_i\in W$
image of $\bs_i$ (cf. \cite[0.1(f)]{BZ}).

A consequence is that the cyclotomic Hecke algebra specializes to the group
algebra of $W$ under the map $u_{s,j}\mapsto\exp(2\pi ij/d_s)$.

Explicit presentations of the form~(\ref{presentation}) of $B(W)$ and hence
of $\cH(W,\bu)$ are known for all irreducible reflection groups, see
Brou\'e--Malle--Rouquier \cite{BMR} and the references given there,
Bessis--Michel \cite{BMi} and Bessis \cite[Th.~0.6]{Be}.
\par
The properties of cyclotomic Hecke algebras have been studied due to their
(conjectured) role in the representation theory of finite reductive groups.
Nevertheless, several important questions remain open at present, or have been
settled only for some of the irreducible reflection groups. We recall them in
Section~\ref{sec:conj}.
\par
Apart   from  these  structural  problems,  there  are  questions  of  a  more
computational  nature which need  to be settled.  We would like  to be able to
write  down  an  explicit  $A$-basis  of  $\cH(W,\bu)$,  with  known structure
constants.  Furthermore,  we  would  like  to  know  explicit  models  for all
irreducible  representations of $\cH(W,\bu)$. Again,  these two questions have
been  solved  for  the  imprimitive  reflection  groups (\cite{AK,MM}). In the
present  paper,  we  solve  these  computational  problems  for  the primitive
irreducible  reflection groups of  dimension at most~3,  which only leaves the
five groups $G_{29},G_{31},G_{32},G_{33}$ and $G_{34}$ (in Shephard and Todd's
notation for the irreducible reflection groups) to be considered.
\par
It easy to see that the reflection representation $V$ of $W$ can be realized
over the field $K_W$ generated by the traces of the elements of $W$ on $V$.
It is a theorem of Benard and Bessis that {\em all} representations of $W$
can be realized over $K_W$.
\par
It has been shown in \cite{MaR} that assuming Conjecture~\ref{conj}(a) below,
the characters of $\cH(W,\bu)$ take their values (on any basis of $\cH(W,\bu)$
consisting of images of a subset of $B(W)$) in the field
$K_W(\bu_{s,i}^{1/e})_{s,i}$,  where $e$ is the order of the group of roots of
unity  in $K_W$. A consequence of our results here is that, whenever we can
compute them, the representations of $\cH(W,\bu)$ have a model where the
matrices for the generators $\bs_i$ have entries in the field generated by the
corresponding character values.

\section{Some conjectures} \label{sec:conj}

We start by recalling some basic conjectures on the structure and representation
theory of cyclotomic Hecke algebras. The most basic conjecture states:

\begin{conj}   \label{conj0}
 Let $W$ be an irreducible complex reflection group, $K=\Frac(A)$. Then:
 \begin{itemize}
  \item[(a)] $\cH(W,\bu)\otimes_A K$ has dimension $|W|$.
  \item[(b)] There exist irreducible non-isomorphic representations $\rho_i$
   of $\cH(W,\bu)$ over $A$ such that $\sum_i\dim(\rho_i)^2=|W|$.
 \end{itemize}
\end{conj}

Part~(a) is known to hold for the infinite series by work of Ariki \cite{Ar} and
Brou\'e--Malle \cite{BroMa}, and
for the 2-dimensional primitive groups by Etingof--Rains \cite{ER}.
Our methods can prove (b) in some cases, which shows that the dimension is at
least that big, but we obtain no information on an upper bound. 
But, assuming these weak statements, we will derive the validity of an 
important stronger assertion.
For this, let now $W$ be irreducible. Then it is known by \cite[remark 12.4]{Be}
that, excepted possibly for the case of $G_{31}$,
the center of $B(W)$ is cyclic, generated by some element $\bz$. We set
$\bpi=\bz^{|ZW|}$ (an element of the pure braid group $\pi_1(V^\reg,x_0)$).

\begin{conj}   \label{conj}
 Let $W$ be an irreducible complex reflection group. Then:
 \begin{itemize}
  \item[(a)] $\cH(W,\bu)$ is free over $A$ of rank $|W|$,
  \item[(b)] $\cH(W,\bu)$ carries a non-degenerate symmetrizing form
   $t:\cH(W,\bu)\rightarrow A$ which makes it into a symmetric algebra, and
   such that
   \begin{equation}\label{tracecond}
     t(T_{b^{-1}})^\vee=t(T_{b\bpi})/t(T_\bpi)\qquad\text{for all }b\in B(W),
   \end{equation}
  where we denote by $b\mapsto T_b$ the natural map from $B(W)\to\cH(W,\bu)$
  and $x\mapsto x^\vee$ is the automorphism of $A$ given by $\bu\mapsto\bu^{-1}$.
 \end{itemize}
\end{conj}

Once~Conjecture~\ref{conj}(a) has been established, it follows from Tits'
deformation theorem that $\cH(W,\bu)$ is a deformation of the group algebra of
$W$, that is, it becomes isomorphic to the group algebra over a suitable finite
extension of the field of fractions of $A$.  

It was shown in \cite[2.1]{BMM} that assuming~(a), there is at
most one symmetrizing trace on $\cH(W,\bu)$ satisfying~(b) which specializes
to the canonical trace on $\CC W$.

Given  a split semi-simple  symmetric algebra $H$  with a symmetrizing form $t$
such  that $t(1)=1$,  we define  the {\em  Schur element} $S_\chi$ attached to
$\chi\in\Irr(H)$ by the property that
\begin{equation}  \label{schur}
  t(x)=\sum_{\chi\in\Irr(H)}\chi(x)/S_\chi\qquad\text{ for all }x\in H.
\end{equation}

Let us denote by $\bs\mapsto T_\bs$ the  natural map $B(W)\to\cH(W,\bu)$. In
\cite{MaD,MaG}, assuming Conjecture~\ref{conj}(a) (which implies that
$\cH(W,\bu)$ is split semi-simple over a suitable extension of $A$), it was
shown that for all exceptional complex reflection groups there is a unique
symmetrizing trace such that $t(T_x)=0$ for
$x\in E\setminus\{1\}$, where $E$ is a subset of $B(W)$ such that
\begin{itemize}
 \item all character values on $\{T_x\mid x\in E\}$ could be determined.
 \item equations~(\ref{schur}) for $x\in E$ are sufficient in number to have
  a unique solution. For instance, it is enough that the image of $E$
  in $W$ intersects all conjugacy classes.
\end{itemize}
Moreover, the corresponding Schur elements $S_\chi$ were determined in all
cases.
When specializing the Hecke algebra to the group algebra $\CC W$, $t$
specializes to the canonical trace $t_W$ on $\CC W$ given by
$t_W(w)=\delta_{w,1}$, so the Schur elements computed in \cite{MaG}
specialize to $|W|/\chi(1)$. 

We fix this symmetrizing form $t$ described above. All of our computational
verifications will depend on a suitable choice of basis for the cyclotomic
Hecke algebra. 

\begin{lem}   \label{lem:base}
 Assume Conjecture~\ref{conj0}. Let $C\subseteq\cH(W,\bu)$ be of cardinality
 $|W|$ and specializing to $W\subseteq \CC W$ under the specialization of
 $\cH(W,\bu)$ to the group algebra. Then $C$ is a $K$-basis of
 $\cH(W,\bu)\otimes_A K$.
\end{lem}

\begin{proof}
Indeed, by~(\ref{schur}) the matrix $M:=t(xy)_{x,y\in C}$ has entries in the
localization of $A$ at the collection of the Schur elements. Since the
specialization of Schur elements is non-zero, we may specialize $M$ to obtain
the corresponding matrix $M:=t_W(vw)_{v,w\in W}$ for $W$, which is a
permutation matrix. Thus, $\det(M)$ is non-zero, and hence $C$ is $K$-linear
independent. The claim then follows from Conjecture~\ref{conj0}.
\end{proof}

An obvious way to construct a set $C$ as above is by lifting the elements of
$W$ to $B(W)$. We are looking for lifts which satisfy an additional
property with respect to $t$:

\begin{conj}\label{conj2}
 There exists a section $W\to \bW$, $w\mapsto\bw$, of $W$ in $B(W)$ such that
 $\bW\owns 1$, and such that for any $\bw\in\bW\setminus\{1\}$ we have
 $t(T_\bw)=0$.
\end{conj}

According to Lemma~\ref{lem:base}, $\{T_\bw\mid w\in\bW\}$ is a $K$-basis of
$\cH(W,\bu)\otimes_A K$. Note, however, that in general it will not necessarily
be an $A$-basis of $\cH(W,\bu)$. Now (\ref{tracecond}) and
Conjecture~\ref{conj2} are related as follows:

\begin{prop}
 Assume Conjecture~\ref{conj2}. If either all irreducible representations of
 $\cH(W,\bu)$ have models over $A$, or else $\{T_\bw\mid w\in\bW\}$ is an 
 $A$-basis of $\cH(W,\bu)$ then property~(\ref{tracecond}) is equivalent to:
 \begin{equation}\label{tracecond1}
  \text{for any $\bw\in\bW-\{1\}$ we have }t(T_{\bw^{-1}\bpi})=0.
 \end{equation}
\end{prop}

\begin{proof}
Using equation (\ref{schur}) for $x=T_b^{-1}$ and $x=T_{b\bpi}$,
condition~(\ref{tracecond}) reads
$$t(T_\bpi)\sum_{\chi\in\Irr(\cH(W,\bu))} \frac{\chi(T_b^{-1})^\vee}
  {S_\chi^\vee}=
  \sum_{\chi\in\Irr(\cH(W,\bu))} \omega_\chi(T_\bpi)\frac{\chi(T_b)}{S_\chi}$$
where $\omega_\chi$ is the central character of $\chi$. Under the standard
specialization $\vhi:A\rightarrow\CC$, $u_{s,j}\mapsto\exp(2\pi ij/d_s)$, we
obviously have the following compatibility with complex conjugation:
$\vhi(a^\vee)=\overline{\vhi(a)}$ for all $a\in A$. Thus
$$\vhi(\chi(T_b^{-1})^\vee)=\overline{\vhi(\chi(T_b^{-1}))}
  =\overline{\chi(b^{-1})}=\chi(b^{-1})=\vhi(\chi(T_b))$$
whence $\chi(T_b^{-1})^\vee=\chi(T_b)$ (note that by our assumptions all
character values $\chi(T_b)$ lie in $A$). Our first equation then reads 
$$t(T_\bpi)\sum_{\chi\in\Irr(\cH(W,\bu))} \frac{\chi(T_b)}{S_\chi^\vee}=
  \sum_{\chi\in\Irr(\cH(W,\bu))} \omega_\chi(T_\bpi)\frac{\chi(T_b)}{S_\chi},$$
which is a linear condition in $T_b$. Thus 
\begin{itemize}
 \item if it holds for the image of $B(W)$ it holds for any element of
  $\cH(W,\bu)$;
 \item it is sufficient to check it for a basis of $\cH(W,\bu)\otimes_AK$.
\end{itemize}
Note that $\{T_\bw^{-1}\mid\bw\in\bW\}$ is still a basis of $\cH(W,\bu)\otimes_AK$
since it is the image of $\{T_\bw\mid\bw\in\bW\}$ by the anti-automorphism
$a_1$ of \cite[1.26]{BMM}. Writing the condition on  this basis we get
$$t(T_\bw)^\vee=\frac{t(T_{\bw^{-1}\bpi})}{t(T_\bpi)}.$$
This holds trivially for $\bw=1$, and for the others $t(T_\bw)^\vee=0$ whence
the result.
\end{proof}

\begin{defn} We say that a section $\bW$ is {\em good} if 
for any $\bw\in\bW-\{1\}$ we have $t(T_\bw)=t(T_{\bw^{-1}\bpi})=0$,
and the matrix $\{t(T_{\bw\bw'})\}_{\bw,\bw'\in\bW}$ is in $\GL_{|W|}(A)$.
\end{defn}

The notion of good section is the tool which will allow us to prove
conjecture \ref{conj} in quite a few cases, using the next proposition:
\begin{prop}\label{good}
Assume that $\bW$ is a good section and that for a generating set $\bS$ of
$B(W)$ we have for all $\bs\in\bS, \bw,\bw'\in \bW$ that 
$t(T_{\bs\bw\bw'})\in A$. Then $\cH(W,\bu)$ satisfies conjecture \ref{conj},
and $\{T_\bw\mid \bw\in \bW\}$ is an $A$-basis of $\cH(W,\bu)$.
\end{prop}
\begin{proof}
Let $M$ be the matrix $\{t(T_{\bw\bw'})\}_{\bw,\bw'\in\bW}$. From the
assumption $M\in \GL(A)$ it follows that the dual basis $T'_\bw$ of
$T_\bw$ with respect to $t$ lies in $A[T_\bw]_{\bw\in \bW}$.
It follows that $h\in \cH(W,\bu)$ lies in $A[T_\bw]_{\bw\in \bW}$ if and
only if for any $\bw$ we have $t(h T_\bw)\in A$; indeed
the coefficient of $h$ on $T_\bw$ is $t(hT'_\bw)$ which is in $A$ if all
the $t(h T_\bw)$ are in $A$.

Thus the condition in the statement shows that 
$T_{\bs\bw}\in A[T_\bw]_{\bw\in \bW}$, i.e. that $T_\bw$ is an $A$-basis.
\end{proof}

\section{Imprimitive groups} \label{sec:imprimitive}

Before turning to the main subject of the present paper, the exceptional
complex reflection groups, we recall the current situation for the infinite
series, that is, the imprimitive groups and the symmetric groups.
Conjecture~\ref{conj}(a) has been verified in these cases by
Ariki--Koike \cite{AK}, Brou\'e--Malle \cite{BroMa} and Ariki \cite{Ar}.
The properties of a symmetrizing form on $\cH(G(de,e,r),\bu)$ have been
investigated in Malle--Mathas \cite{MM}. It is not clear, though, that it
satisfies the additional properties mentioned in Conjecture~\ref{conj}.
Conjecture~\ref{conj2} has been verified for $G(d,1,r)$ by Bremke--Malle
\cite{BrMa}.

Explicit models for the irreducible representations of the generic cyclotomic
Hecke algebra for the imprimitive complex reflection group $G(d,1,r)$ have
been given by Ariki--Koike \cite{AK}, and have been extended to $G(de,e,r)$ 
by Ariki \cite{Ar}. However, these models are over $K_W(u^{1/e}_{s,i})_{s,i}$.

Models for the case $G(d,1,r)$ are known over $A$, using the fact that this is a
cellular algebra, and that the generators act with coefficients in $A$
on a cellular basis. For example this can be seen from Dipper--James--Mathas
\cite{DJM}. Let the generators of $\cH(W,\bu)$ correspond to the diagram
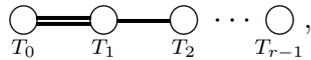
\newcommand{\nnode}[1]{{\kern 5pt\hbox to
0pt{\hss{$\mathop\bigcirc\limits_{#1}$}\hss}\kern 5pt}}
\newcommand{\sbar}{{\vrule width20pt height3pt depth-2pt}}
\newcommand{\dbar}{{\rlap{\vrule width20pt height2pt depth-1pt} 
                 \vrule width20pt height4pt depth-3pt}}
$$\nnode{T_0}\dbar\nnode{T_1}\sbar\nnode{T_2}\ \cdots\ \nnode{T_{r-1}}\kern 4pt,$$
where $T_0$ has parameters $Q_1=u_{0,0},\ldots,Q_d=u_{0,d-1}$ and the $T_i$ 
($i\ne 1$) have parameters $q=u_{1,0}, -1=u_{1,1}$.
Then the action of the $T_i$ in a cellular basis is given by 
\cite[3.15 and 3.18]{DJM}, while the action of $T_0$ is given by
\cite[3.20]{DJM} (note that only the term $x_1$ of loc. cit. subsists in the
model for the representation $\lambda$).
\par
We are not aware of similarly nice integral/rational models for the
representations of $\cH(G(de,e,r),\bu)$, where $e>1$.

\section{Two-dimensional primitive groups} \label{sec:2-dim}

In this section we describe the construction of models for the irreducible
representations of the Hecke algebras $\cH(W,\bu)$, where $W$ is a primitive
2-dimensional reflection group, so one of the groups $G_4,\ldots,G_{22}$.
We first describe several reductions. \par
\noindent {\bf Step 1:} It is sufficient to find models in the case of
$G_7,G_{11}$ and $G_{19}$.\par
The  braid groups of $G_7$, $G_{11}$ and $G_{19}$ are isomorphic to the same
group
$$B:=\langle \bs_1, \bs_2, \bs_3 \mid
     \bs_1\bs_2\bs_3=\bs_2\bs_3\bs_1=\bs_3\bs_1\bs_2 \rangle$$
(see \cite[\S5]{BMR}).
Let $\bu=(x_1,x_2;y_1;y_2;y_3;z_1,\ldots,z_k)$, where $k=3$ for $G_7$ (resp.
4, 5 for $G_{11}$, $G_{19}$). The cyclotomic Hecke algebra $\cH(G,\bu)$ of
$G_7$ (resp. $G_{11}$, $G_{19}$) is the quotient of the group algebra of $B$
over $\ZZ[\bu,\bu^{-1}]$ by the relations
$$(\bs_1-x_1)(\bs_1-x_2)=0,\quad(\bs_2-y_1)(\bs_2-y_2)(\bs_2-y_3)=0,\quad
  \prod_{i=1}^{i=k}(\bs_3-z_i)=0.$$
In turn the Hecke algebras for $G_4$ to $G_6$ are subalgebras of suitable
partial specializations of that for $G_7$ (the same holds for $G_8$ to
$G_{15}$ with respect to $G_{11}$ and for $G_{16}$ to $G_{22}$ with respect to
$G_{19}$) (see \cite[Prop.~4.2]{MaD}). More precisely, in each case, these
algebras are generated by suitable conjugates of a subset of the generators
(or of some power of them), while the other generators are specialized to the
group algebra. The necessary generators are collected in Table~\ref{tab:gens}.

\begin{table}[htbp]
\caption{Generators for Hecke algebras of 2-dimensional primitive groups}
  \label{tab:gens}
$$\begin{array}{c|c}
 W& \text{generators of }\cH(W)\cr
\hline
 G_4, G_8, G_{16}& \bs_3, \lexp{\bs_1}\bs_3\cr
 G_5, G_{10}, G_{18}& \bs_2, \bs_3\cr
 G_6, G_9, G_{17}& \bs_1, \bs_3\cr
 G_{14}, G_{21}& \bs_1, \bs_2\cr
 G_{12}, G_{22}& \bs_1,\lexp{\bs_2}\bs_1,\bs_1^{\bs_2}\cr
 G_{20}& \bs_2, \lexp{\bs_1}\bs_2\cr
 G_{13}& \bs_3^2,\bs_1, \bs_1^{\bs_2}\cr
 G_{15}& \bs_1,\bs_2, \bs_3^2\cr
\end{array}$$
\end{table}

Moreover, each irreducible representation of the Hecke algebra of any of the
groups $G_4,\ldots,G_{22}$ can be obtained as the restriction of an
irreducible representation of the Hecke algebra of one of $G_7,G_{11}$ or
$G_{19}$. It follows that it is sufficient to determine the representations
of the Hecke algebras of $G_7,G_{11}$ and $G_{19}$ to determine the
representations of the Hecke algebras of all 2-dimensional primitive reflection
groups.
\vskip 1pc

\noindent {\bf Step 2:} It is sufficient to compute irreducible representations
of $B$ of dimension $2\le d\le 6$, with an additional condition on the
eigenvalues of the generators. \par
The  irreducible representations  of $G_7$  have dimension 1,2 or 3, those of
$G_{11}$  dimension 1 to 4 and those of  $G_{19}$ dimension 1 to 6. It follows
that  any  2-dimensional  representation  of  $B$  gives  a  representation of
$\cH(W,\bu)$  where  $W$  is  any  of  $G_7,G_{11},G_{19}$;  any 3-dimensional
representation  of $B$ where  $\bs_1$ has only  2 distinct eigenvalues gives a
representation  of the same algebras;  any 4-dimensional representation of $B$
where  $\bs_1$ has only 2 distinct eigenvalues and $\bs_2$ has only 3 distinct
eigenvalues  gives  a  representation  of  $\cH(W,\bu)$  where  $W$  is any of
$G_{11},G_{19}$;  finally any 5  or 6-dimensional representation  of $B$ where
$\bs_1$ has  only 2  distinct eigenvalues,  $\bs_2$ has only 3 distinct
eigenvalues and $\bs_3$ has only 5 distinct eigenvalues gives a representation
of $\cH(G_{19},\bu)$.
\vskip 1pc

\noindent {\bf Step 3:} It is sufficient to compute \emph{one} irreducible
representations of $B$ in each dimension $2\le d\le 6$.\par
For each dimension (from 1 to 6) and each $W\in\{G_7,G_{11},G_{19}\}$, the
irreducible representations of $\cH(W,\bu)$ (up to isomorphism) form a single
orbit under the Galois automorphisms corresponding to permuting the
$x_i$, the $y_i$, the $z_i$ among themselves. It transpires that we need just
to find one representation of $B$ of the right dimension with the required
number of eigenvalues.\par

\vskip 1pc
It turns out that one can find such representations of $B$ by
matrices of the form
\def\addots{%
	\mathinner{\mkern1mu\raise1pt\vbox{\kern7pt\hbox{.}}\mkern2mu
    \raise4pt\hbox{.}\mkern2mu\raise7pt\hbox{.}\mkern1mu}}
$$
\bs_1\mapsto
\begin{pmatrix}
*&\ldots&\ldots&*\\
0&\ddots& &\vdots\\
\vdots&\ddots&\ddots&\vdots\\
0&\ldots&0 &*\\ 
\end{pmatrix},\ \ 
\bs_2\mapsto
\begin{pmatrix}
*&\ldots&\ldots&*\\
\vdots& &\addots&0\\
\vdots&\addots&\addots&\vdots\\
*&0&\ldots&0\\ 
\end{pmatrix},\ \ 
\bs_3\mapsto
\begin{pmatrix}
0&\ldots&0 &*\\ 
\vdots&\addots&\addots&\vdots\\
0&\addots& &\vdots\\
*&\ldots&\ldots&*\\
\end{pmatrix}.
$$

A solution for the 2-dimensional representation is
$$\bs_1\mapsto
\begin{pmatrix}x_1&\frac{y_1+y_2}{y_1y_2}-\frac{(z_1+z_2)x_2}r\\0&x_2
\end{pmatrix},\ 
\bs_2\mapsto
\begin{pmatrix} y_1+y_2&1/x_1\\ -y_1y_2x_1&0 \end{pmatrix},\ 
\bs_3\mapsto
\begin{pmatrix} 0&\frac {-r}{y_1y_2x_1x_2}\\ r&z_1+z_2 \end{pmatrix}$$
where $r=\sqrt{x_1x_2y_1y_2z_1z_2}$. Note that the irrationality $r$ occurring
in the matrices is necessary \cite[Tab.~8.1]{MaR}.

A solution for the 3-dimensional representation is
$$\bs_1\mapsto
\begin{pmatrix}
x_1& 0&((z_2z_3+z_1z_3+z_1z_2)x_2x_1r^{-1}
      -\frac{(y_3+y_1+y_2)r}{y_1y_2y_3})z_1^{-1}\\
0& x_1& -r(y_1y_2y_3z_1)^{-1}\\
0& 0& x_2\\
\end{pmatrix},
$$
$$
\bs_2\mapsto
\begin{pmatrix}
y_1+y_2+y_3-r(x_1z_1)^{-1}&a{z_1}^{-1}& -1\\
1& r(x_1z_1)^{-1}& 0\\
y_1y_2y_3x_1z_1r^{-1}& 0& 0\\
\end{pmatrix},
$$
$$
\bs_3\mapsto
\begin{pmatrix}
0& 0& z_2z_3x_1r^{-1}\\
0& z_1& 0\\
-r{x_1}^{-1}& a& z_3+z_2\\
\end{pmatrix},
$$

where $a=(y_3+y_1+y_2)rx_1^{-1}-(y_1y_3+y_1y_2+y_3y_2)z_1
+y_1y_2y_3(x_1z_1^2-x_2z_2z_3)r^{-1}$ and

where $r=\sqrt[3]{x_1^2x_2y_1y_2y_3z_1z_2z_3}$.

A solution for the 4-dimensional representation is
$$
\bs_1\mapsto
\begin{pmatrix}
x_1& 0&x_1a-x_1x_2y_1\frac br&x_1(1+\frac{y_1}{y_3})-
\frac r{y_3}\sum_i\frac1{z_i}\\
0& x_1& \frac1{y_1}+\frac1{y_2}& -\frac{x_2}{r^3}\\
0& 0& x_2& 0\\
0& 0& 0& x_2\\
\end{pmatrix},
$$
$$
\bs_2\mapsto
\begin{pmatrix}
y_3+y_1&x_1y_1y_2a&y_1a&y_1\\
0&y_1+y_2&1/x_1&0\\
0&-x_1y_1y_2&0&0\\
-y_3&0&0&0\\
\end{pmatrix},
$$
$$
\bs_3\mapsto
\begin{pmatrix}
0&0&0&-r/(y_3x_2)\\
0&0&-r/(y_2x_1x_2y_1)&0\\
0&r&0&1/r^2\\
r/(x_1y_1)&-ra&b&\sum_i z_i\\
\end{pmatrix},
$$
where
$$a=x_1x_2y_1y_2\prod_j z_j(\sum_i\frac 1{z_i})-r^2\sum_{i=1}^4 z_i,\quad
  b=x_1x_2y_1(y_2+y_3)\prod_i z_i-r^2\sum_{i<j}z_iz_j$$
and where $r=\sqrt[4]{x_1^2x_2^2y_1y_2y_3^2z_1z_2z_3z_4}$.

We refer to the GAP-part of the Chevie system \cite{MChevie} for
solutions for the 5-dimensional and 6-dimensional representations of $B$.
By our above reductions, this completes the construction of the irreducible
representations of all cyclotomic Hecke algebras attached to 2-dimensional
exceptional complex reflection groups.

\section{Hensel lifting and Pad\'e approximation} \label{sec:hensel}

We now describe computational techniques used to obtain models for
irreducible representations of Hecke algebras for higher dimensional
primitive complex reflection groups. It is not an algorithm in the sense
that it does not always succeed, but in the case of one-parameter algebras,
it turned out to have a good rate of success. It consists of Hensel lifting
representations of $W$ to $\cH(W)$, combined with Pad\'e approximation.

We note that for groups generated by true reflections which are all
conjugate, such as $G_{24}$, $G_{27}$, $G_{29}$, $G_{31}$, $G_{33}$ and
$G_{34}$, there are only two parameters $u_{s,0}$ and $u_{s,1}$, and with the
usual normalization $u_{s,1}=-1$ (corresponding to replacing $\bs$ by
$-\bs/u_{s,1}$) there is only one parameter $q=-u_{s,0}/u_{s,1}$. We will
write $\cH(W,q)$ for such an algebra.

\subsection{Hensel lifting representations of $W$}
We start with a presentation of $B(W)$, of the form
$$\langle \bs_1,\ldots,\bs_n\mid
  p_j(\bs_1,\ldots,\bs_n)=q_j(\bs_1,\ldots,\bs_n)\rangle$$
as explained in (\ref{presentation}).

If $\rho_q:\cH(W,q)\rightarrow M_{l\times l}(\CC(q))$ is a representation of
$\cH(W,q)$ over $\CC(q)$, and if $M_i=\rho_q(\bs_i)$, the idea consists 
in writing $M_i$ as a series in the variable $r:=q-1$. If we have such a
development $M_i=M_i^{(0)}+r M_i^{(1)}+r^2 M_i^{(2)}+\ldots$ where
$M_i^{(j)}\in M_{l\times  l}(\CC)$, then $M_i^{(0)}$ is the specialization
$\rho_1$ of $\rho_q$ at $q=1$, a representation of $W$.

Conversely,  if we start with a representation  $\rho_1$ of $W$, we may try to
extend it to a representation of $\cH(W,q)$ by solving the system of equations
$$(M_i+1)(M_i-r-1)=0\qquad\text{and}\quad
  p_j(M_1,\ldots,M_n)=q_j(M_1,\ldots,M_n),$$
where $M_i=M_i^{(0)}+r M_i^{(1)}+r^2 M_i^{(2)}+\ldots\in
M_{l\times l}(\CC((r)))$ are formal power series
with $M_i^{(0)}=\rho_1(\bs_i)$.  The point here is that if we already have a
solution $M_i^{(j)}$ for  $j<j_0$ (where  $j_0\ge1$)  then the  equations  for
$M_i^{(j_0)}$  form a system of linear equations,  which, if $v$ is the vector
of all the entries  $(M_i^{(j_0)})_{k,l}$ of the  matrices $M_i^{(j_0)}$, has
the form $\Lambda v=N_{j_0}$, for a matrix $\Lambda$ which is independent of
$j_0$.

Unfortunately the matrix $\Lambda$ does not have full rank in practice. To try
to  solve the above system, we choose  for $j_0=1$ a matrix $\Lambda'$ of full
rank  extending  $\Lambda$,  and  then  solve  iteratively  each step $j_0$ by
setting  $v=\Lambda^{\prime-1}  N_{j_0}$.  We  thus  get  a  representation of
$\cH(W,q)$ with coefficients in $\CC[[r]]$; actually in $K[[r]]$ if $K$ is the
field where the entries of the matrices $M_i^{(0)}$ lie.

In our computations it happened quite often that this representation is
actually over $K(r)=K(q)$. This is the point of the method. To increase the
probability that this happens, we found a number of heuristics:

\begin{itemize}
 \item As equations added to $\Lambda$ to make $\Lambda'$, we first try to add
  equations specifying that undetermined entries in $M_i^{(1)}$ where the 
  corresponding entry in $M_i^{(0)}$ is $0$ should also be $0$.
 \item  If the chosen model of $\rho_1$ given by the matrices $M_i^{(0)}$ does
  not give good results, change the model randomly (but such that it is still
  \lq simple\rq) until a better result occurs.
\end{itemize}

\subsection{Recognizing the entries}
To recognize that the obtained series $M_i\in M_{l\times l}(K[[r]])$ lies in
$K(r)$,  we use {\em  Pad\'e approximation}: if  a series $h\in K[[r]]$, which
can  be assumed to have  a non-zero constant coefficient,  is the expansion of
$f/g\in  K(r)$ where $f,g\in K[r]$  are of degree less  than $d$ and $g(0)=1$,
then $f$ and $g$ are determined by solving linear equations involving only the
first  $2d$ terms of  $h$. If these  linear equations have  a solution, we say
that  $f/g$ is a Pad\'e approximant of $h$. \par
This is applied to the (approximate) entries of $M_i$ as follows: We compute
Pad\'e approximants for increasing $d$, until they become stationary, which
generally means that we have found a solution in $K(r)$.

Note that it is very easy afterwards to check whether the result of our
computations does indeed define a representation of $\cH(W,q)$, by just
evaluating the relations of $\cH(W,q)$.

The representations of $\cH(W,q)$ are in general not defined over $\CC(q)$ but
over  $\CC(q^{1/e})$ where $e$ is the order of  the group of roots of unity in
$K_W$. To handle this case it is sufficient to take $r:=q^{1/e}-1$ as a variable
and apply the same construction.

\subsection{Finding good models for representations of $W$}
To start the process we needed to get a complete set of models for the
irreducible representations of $W$. For this, we used the following techniques:

\begin{itemize} 
 \item Get new representations from known representations by tensoring by
  linear characters and applying Galois actions.
 \item Get new representations as Schur functors of known representations
  (when such Schur functors happen to give an irreducible representation; 
  an example is that the exterior powers $\Lambda^i V$ are always irreducible
  and the symmetric square $S^2V$ is
  irreducible if $W$ is not real). We have written a Chevie-program
  to compute general Schur functors to do this.
\end{itemize}

For  example only 7 of the 90 representations of dimension $\le 60$ of
$G_{32}$ cannot be obtained by the above process starting from the
reflection representation. To get the remaining representations, we need one
more technique:

\begin{itemize} 
 \item  Obtain the desired representation as a component of multiplicity $1$ in
  the tensor product of two known representations. Then compute a model by
  explicitly computing the projector on the desired isotypic component.
\end{itemize}

It  turns out  that all irreducible representations of exceptional complex
reflection groups can be obtained from the reflection
representation applying  these three steps.  To compute the projector on the
isotypic  component, we explicitly compute the image of the class sums of $W$
in  the representation,  by enumerating  the elements  of $W$  as words in the
generators  and computing their images. We have carried out this computation
for all groups considered excepted  $G_{34}$ where  this would need to add
together billions of matrices of rank several tens of thousands, which is a
larger computation that those we have attempted.

For the questions to be considered below, but also for other computational
purposes, it is desirable
to have a model with few non-zero entries, which are integral if possible.
We try to achieve this by performing suitable base changes on the first model.
A good heuristic which tends to simplify the model a lot is to use a basis
consisting of one-dimensional intersections of eigenspaces of the
matrices $M_i^{(0)}$.

An example of a representation obtained by the methods of this section and
that we could not obtain in another way is the representation $\phi_{8,5}$ of
$\cH(G_{24},\{x,y\})$ (here $v=\sqrt{-xy}$, and $\bs,\bt,\bu$ are the
generators in the presentation $P_1$ given below in  \ref{subsec:G24}):

$$ \bs\mapsto
\begin{pmatrix}
.&.&.&.&.&.&.&-x\\
.&x+y&.&.&y&.&.&.\\
.&.&x&-vy+xy&.&.&-x^2&.\\
.&.&.&y&.&.&.&.\\
.&-x&.&.&.&.&.&.\\
.&.&.&x&.&x&-v-y&.\\
.&.&.&.&.&.&y&.\\
y&.&.&.&.&.&.&x+y\\
\end{pmatrix}
$$
$$\bt\mapsto
\begin{pmatrix}
x&.&.&v&.&.&.&-y\\
.&x&.&v&x&.&.&.\\
.&.&x+y&.&.&.&-xy&.\\
.&.&.&y&.&.&.&.\\
.&.&.&.&y&.&.&.\\
.&.&-1&x&-v&x&x&v\\
.&.&1&.&.&.&.&.\\
.&.&.&.&.&.&.&y\\
\end{pmatrix}
$$
$$\bu\mapsto
\begin{pmatrix}
y&.&.&.&.&.&.&.\\
.&x&.&.&x&.&-v&.\\
-xy&.&x&.&-vy&vy&vy-xy-x^2&.\\
.&.&.&x&.&-y&-v-y&.\\
.&.&.&.&y&.&.&.\\
.&.&.&.&.&y&.&.\\
.&.&.&.&.&.&y&.\\
x&.&.&.&.&.&x&x\\
\end{pmatrix}
$$

Another example is the representation $\phi_{8,6}$ of
$\cH(G_{27},\{x,y\})$ where again $v=\sqrt{-xy}$, and $\bs,\bt,\bu$ are the
generators in the presentation $P_1$ given below in  \ref{subsec:G27}):

$$\bs \mapsto
\begin{pmatrix}
x&.&-y&.&.&.&v-y\frac{1+\sqrt 5}2&.\\
.&x+y&.&y&.&.&.&.\\
.&.&y&.&.&.&.&.\\
.&-x&.&.&.&.&.&.\\
.&.&.&.&.&.&.&x\\
.&.&v+x\frac{1+\sqrt 5}2&.&.&x&x&.\\
.&.&.&.&.&.&y&.\\
.&.&.&.&-y&.&.&x+y\\
\end{pmatrix}
$$

$$\bt \mapsto
\begin{pmatrix}
x&.&-y&y&.&v&v&-y\\
.&x&-x\frac{3+\sqrt 5}2&x&.&.&.&.\\
.&.&y&.&.&.&.&.\\
.&.&.&y&.&.&.&.\\
.&.&x&.&x&.&.&y\\
.&.&.&.&.&x+y&y&.\\
.&.&.&.&.&-x&.&.\\
.&.&.&.&.&.&.&y\\
\end{pmatrix}
$$

$$\bu \mapsto
\begin{pmatrix}
.&.&-x&.&.&.&.&.\\
.&x&.&x&.&.&-y\frac{3+\sqrt 5}2&.\\
y&.&x+y&.&.&.&.&.\\
.&.&.&y&.&.&.&.\\
.&.&.&.&y&.&.&.\\
v&.&v&-y&v&x&x&.\\
.&.&.&.&.&.&y&.\\
.&.&.&.&-x&.&v&x\\
\end{pmatrix}
$$

\section{Presentations of $B(G_{24})$ to $B(G_{34})$}\label{sec:Presentations}

In  \cite{BMi} we considered presentations of exceptional complex braid groups,
and  proposed several presentations for $B(G_{24})$, $B(G_{27})$, $B(G_{29})$,
$B(G_{33})$  and $B(G_{34})$. In the  context of our current  work, it will be
important  to consider  alternative presentations,  since at  least two of the
properties  we consider (the existence of  $W$-graphs and the vanishing of the
trace on minimal length elements) turn out to depend on the presentation, with
each  time a presentation faring better than  the others with respect to these
properties.  We have noticed a framework  in which these various presentations
fit and can be systematically recovered.

Since  the  groups  above  are  well-generated,  they  have  a  unique maximal
reflection  degree $h$ called the {\em Coxeter number} of $W$. We proved in
\cite{BMi} that in each case, the product $\delta=\bs_1\cdots\bs_n$ of the
generators of the braid group in a certain order is an $h$-th root of the
generator of the center of the pure braid group. The image $c$ of $\delta$ in
$W$ is
$e^{2i\pi/h}$-regular  in the sense of Springer, and if choosing for basepoint
a   $e^{2i\pi/h}$-regular  eigenvector  $x$  of   $c$,  the  element  $\delta$
corresponds  to a path joining $x$ to $e^{2i\pi/h}x$ (these two points coincide
in $V^\reg/W$).

We  consider the {\em Hurwitz action}  of the ordinary Artin braid group $B_n$,
the group with presentation
$$\langle\sigma_1,\ldots,\sigma_{n-1}\mid    \sigma_i\sigma_j=\sigma_j\sigma_i
  \text{ if $|i-j|>1$},
  \sigma_i\sigma_{i+1}\sigma_i=\sigma_{i+1}\sigma_i\sigma_{i+1}\rangle$$
on decompositions  $\delta=\bs_1\ldots\bs_n$ of $\delta$, given by
$$\sigma_i:\ (\bs_1,\ldots,\bs_n)\mapsto
  (\bs_1,\ldots,\bs_{i+1},\bs_i^{\bs_{i+1}},\ldots,\bs_n),$$
so that
$$\sigma_i^{-1}:(\bs_1,\ldots,\bs_n)\mapsto
  (\bs_1,\ldots,\lexp{\bs_i}\bs_{i+1},\bs_i,\ldots,\bs_n).$$
We thus obtain new decompositions of $\delta$ into a product of $n$ braid
reflections.

Bessis  has  shown  in  \cite{Be}  that  the  orbit  of  the Hurwitz action on
decompositions  of $\delta$  is finite,  of cardinality  $n! h^n/|W|$. What we
noticed  is that all the presentations of  \cite{BMi} correspond to taking as a
set  of generators the ones  which appear in one  decomposition in the Hurwitz
orbit; in the case of $G_{24}$ and $G_{27}$ any element of a Hurwitz orbit
corresponds up to some permutation to one of the presentations given in
\cite{BMi}; in the other cases some other presentations may appear.

We  give now  the results  in each  case. We  found that  the ``quality'' of a
presentation seems to be correlated to how ``spread out'' their ``Poincar\'e''
polynomial  $\sum_{w\in W}q^{l(w)}$ is (where $l(w)$  is the minimal length in
terms  of the  generators needed  to write  $w$); the  presentations where the
Poincar\'e polynomial has a higher degree are better.

\subsection{Presentations for $B(G_{24})$}\label{subsec:G24}

A Hurwitz orbit of  $\delta$ has 49  elements. Three different presentations
appear  along an orbit. 

\subsubsection*{$P_1$}
The presentation $P_1$ is
$$
  \langle \bs,\bt,\bu\mid
  \bs\bt\bs=\bt\bs\bt, \bt\bu\bt\bu=\bu\bt\bu\bt, \bs\bu\bs=\bu\bs\bu,
  (\bt\bu\bs)^3=\bu\bt\bu(\bs\bt\bu)^2
\rangle
$$
It appears 21 times in a Hurwitz orbit.
Its Poincar\'e polynomial is
\begin{multline*}
q^{15}+3q^{14}+6q^{13}+12q^{12}+27q^{11}+46q^{10}+55q^9+54q^8+44q^7\hfill\\
\hfill+31q^6+ 22q^5+15q^4+10q^3+6q^2+3q+1$$
\end{multline*}

\subsubsection*{$P_2$}
The presentation $P_2$ is
$$
  \langle \bs,\bt,\bu\mid
  \bs\bt\bs\bt=\bt\bs\bt\bs, \bt\bu\bt\bu=\bu\bt\bu\bt, \bs\bu\bs=\bu\bs\bu,
  \bt(\bs\bt\bu)^2=(\bs\bt\bu)^2\bs
\rangle
$$
We get $P_2$ from $P_1$ by taking $\{\bs,\bt\bu\bt^{-1},\bt\}$ as generators.
It appears 21 times in a Hurwitz orbit.
Its Poincar\'e polynomial is
$$q^{13}+4q^{12}+16q^{11}+39q^{10}+56q^9+58q^8+52q^7+42q^6+29q^5+18q^4
+11q^3+6q^2+3q+1$$

\subsubsection*{$P_3$}
The presentation $P_3$ is
$$
  \langle \bs,\bt,\bu\mid
  \bs\bt\bs\bt=\bt\bs\bt\bs, \bt\bu\bt\bu=\bu\bt\bu\bt, 
  \bs\bu\bs\bu=\bu\bs\bu\bs,
  (\bt\bu\bs)^2\bt=(\bs\bt\bu)^2\bs=(\bu\bs\bt)^2\bu
\rangle
$$
We get $P_3$ from $P_1$ by taking
$\{\bt,\bu,\bu^{-1}\bt^{-1}\bs\bt\bu\}$ as generators.
It appears 7 times in a Hurwitz orbit.
Its Poincar\'e polynomial is
$$q^{13}+5q^{12}+12q^{11}+24q^{10}+45q^9+54q^8+59q^7+57q^6+36q^5+21q^4
+12q^3+6q^2+3q+1$$

\subsection{Presentations for $B(G_{27})$}\label{subsec:G27}
A Hurwitz orbit of  $\delta$ has 75  elements. Five different presentations
appear  along an orbit, each 15 times.

\subsubsection*{$P_1$}
The presentation $P_1$ is
$$
  \langle \bs,\bt,\bu\mid
  \bt\bs\bt=\bs\bt\bs, \bu\bs\bu=\bs\bu\bs, \bu\bt\bu\bt=\bt\bu\bt\bu,
  \bu\bt\bu(\bs\bt\bu)^3=(\bt\bu\bs)^3\bt\bu\bt
  \rangle
$$
Its Poincar\'e polynomial is
\begin{multline*}
q^{25}+5q^{24}+12q^{23}+26q^{22}+51q^{21}+88q^{20}+125q^{19}+150q^{18}
 +168q^{17}+191q^{16}+218q^{15}\hfill\\
 \hfill +223q^{14}+200q^{13}+168q^{12}+139q^{11}+114q^{10}+87q^9+62q^8\\
 \hfill +44q^7+31q^6+22q^5+15q^4+10q^3+6q^2+3q+1\\
\end{multline*}

\subsubsection*{$P_2$}
The presentation $P_2$ is
$$
\langle \bs,\bt,\bu\mid
  \bs\bu\bs=\bu\bs\bu, \bs\bt\bs\bt=\bt\bs\bt\bs, \bt\bu\bt\bu\bt=\bu\bt\bu\bt\bu,
  (\bu\bt\bs)^2\bt=\bs(\bu\bt\bs)^2
  \rangle
$$
We get $P_2$ from $P_1$ by taking $\{\bt,\bt\bu\bt^{-1},\bs\}$ as generators.
Its Poincar\'e polynomial is
\begin{multline*}
q^{21}+6q^{20}+22q^{19}+59q^{18}+107q^{17}+152q^{16}+208q^{15}+256q^{14}
+270q^{13}+255q^{12}+218q^{11}\hfill\\
\hfill+177q^{10}+137q^9+100q^8+71q^7+49q^6+32q^5+19q^4+11q^3+6q^2+3q+1\\
\end{multline*}

\subsubsection*{$P_3$}
The presentation $P_3$ is
$$
  \langle \bs,\bt,\bu\mid
  \bs\bt\bs=\bt\bs\bt, \bt\bu\bt\bu\bt=\bu\bt\bu\bt\bu, \bs\bu\bs=\bu\bs\bu,
  \bt\bu\bt\bu(\bs\bt\bu)^2=\bu(\bt\bu\bs)^3
  \rangle
$$
We get $P_3$ from $P_1$ by taking $\{\bs,\bs^{-1}\bu\bs,\bt\}$ as
generators. Its Poincar\'e polynomial is
\begin{multline*}
q^{23}+3q^{22}+6q^{21}+21q^{20}+60q^{19}+121q^{18}+164q^{17}+192q^{16}
+228q^{15}+256q^{14}+245q^{13}\hfill\\
\hfill+210q^{12}+175q^{11}+138q^{10}+106q^9+78q^8+57q^7+38q^6+25q^5
+16q^4+10q^3+6q^2+3q+1\\
\end{multline*}

\subsubsection*{$P_4$}
The presentation $P_4$ is
$$
\left\langle \bs,\bt,\bu\mid
  \begin{aligned}
  &\bs\bt\bs\bt=\bt\bs\bt\bs, \bt\bu\bt\bu\bt=\bu\bt\bu\bt\bu,
   \bs\bu\bs\bu\bs=\bu\bs\bu\bs\bu,\\
  &(\bt\bu\bs)^2\bt=\bs(\bt\bu\bs)^2,
   \bu\bs(\bt\bu\bs)^2=(\bs\bt\bu)^2\bs\bu\\
  \end{aligned}
\right\rangle$$
We get $P_4$ from $P_1$ by taking $\{\bt,\bu,\bu^{-1}\bt^{-1}\bs\bt\bu\}$ as
generators. Its Poincar\'e polynomial is
\begin{multline*}
q^{19}+5q^{18}+16q^{17}+54q^{16}+127q^{15}+211q^{14}+257q^{13}+277q^{12}
+288q^{11}+266q^{10}+217q^9\hfill\\
\hfill+164q^8+117q^7+73q^6+42q^5+23q^4+12q^3+6q^2+3q+1\\
\end{multline*}

\subsubsection*{$P_5$}
Finally $P_5$ just presents the opposite group to $P_2$ (the first 3 relations
are  the same and  each side of  the fourth is  reversed). It is obtained from
$P_2$ by exchanging the generators $\bu$ and $\bt$.

\subsection{Presentations for $B(G_{29})$}
For  $B(G_{29})$, in \cite{BMi}  we considered two  presentations on generators
$\{\bs,\bt,\bu,\bv\}$.   These  presentations   correspond  actually   to  two
different    presentations   of   the    parabolic   subgroup   generated   by
$\{\bt,\bu,\bv\}$ which is of type $B(G(4,4,3))$, so we first describe the
situation for this last group.

For $B(G(4,4,3))$ a Hurwitz orbit of $\delta$ has 32 elements; two
presentations occur along the orbit:

\subsubsection*{$P_1$}
The presentation $P_1$ is
$$
  \langle \bt,\bu,\bv\mid
  \bt\bv\bt=\bv\bt\bv, \bu\bv\bu=\bv\bu\bv, \bt\bu\bt\bu=\bu\bt\bu\bt,
  (\bv\bu\bt)^2=(\bu\bt\bv)^2\rangle
$$
It appears 16 times in a Hurwitz orbit. Its Poincar\'e polynomial is
$$3q^8+13q^7+23q^6+22q^5+15q^4+10q^3+6q^2+3q+1.$$

\subsubsection*{$P_2$}
The presentation $P_2$ is
$$
\langle \bt,\bu,\bv\mid
  \bt\bv\bt=\bv\bt\bv, \bu\bv\bu=\bv\bu\bv, \bt\bu\bt=\bu\bt\bu,
  \bv\bt(\bu\bv\bt)^2=(\bu\bv\bt)^2\bu\bv\rangle
$$
We get $P_2$ from $P_1$ by taking $\{\bv^{-1}\bt\bv,\bu,\bv\}$ as
generators.
It appears 8 times in a Hurwitz orbit. Its Poincar\'e polynomial is
$$11q^8+21q^7+18q^6+15q^5+12q^4+9q^3+6q^2+3q+1.$$

A presentation of $B(G_{29})$ can be obtained in each case by adding one
generator $\bs$ and the extra relations
$\bs\bt\bs=\bt\bs\bt, \bs\bu=\bu\bs, \bs\bv=\bv\bs$.

\subsection{Presentations for $B(G_{33})$ and $B(G_{34})$}\label{subsec:G33}
For  $B(G_{33})$, in \cite{BMi}  we considered two  presentations on generators
$\{\bs,\bt,\bu,\bv,\bw\}$ and for $B(G_{34})$ two presentations on generators
$\{\bs,\bt,\bu,\bv,\bw,\bx\}$.  These  presentations  differ  only on the
parabolic subgroup generated by $\{\bt,\bu,\bv,\bw\}$ which is of type
$B(G(3,3,4))$, so we describe the situation for this last group.

For $B(G(3,3,4))$ a Hurwitz orbit of $\delta$ has 243 elements. Five
different presentations occur along the orbit.

\subsubsection*{$P_1$} The presentation $P_1$ is
$$\left\langle \bt,\bu,\bv,\bw\mid
  \begin{aligned}
  &\bu\bt\bu=\bt\bu\bt,\bv\bt\bv=\bt\bv\bt,
   \bv\bu\bv=\bu\bv\bu,\bv\bw\bv=\bw\bv\bw,\\
  &\bt\bw=\bw\bt, \bu\bw=\bw\bu,(\bv\bt\bu)^2=(\bu\bv\bt)^2\\
  \end{aligned}
\right\rangle$$
It appears 108 times in a Hurwitz orbit. Its Poincar\'e polynomial is
$$8q^{12}+40q^{11}+82q^{10}+108q^9+109q^8+95q^7+79q^6+57q^5+35q^4+20q^3+10q^2+4q+1$$

\subsubsection*{$P_2$} The presentation $P_2$ is
$$\left\langle \bt,\bu,\bv,\bw\mid
  \begin{aligned}
  &\bw\bt\bw=\bt\bw\bt,\bu\bt\bu=\bt\bu\bt,
   \bu\bv\bu=\bv\bu\bv,\bw\bv\bw=\bv\bw\bv,\\
  &\bt\bv=\bv\bt,\bw\bu=\bu\bw,
   \bv(\bw\bt\bu\bv)^2=(\bw\bt\bu\bv)^2\bw\\
  \end{aligned}
\right\rangle$$
It appears 9 times in a Hurwitz orbit. Its Poincar\'e polynomial is
$$q^{12}+20q^{11}+74q^{10}+128q^9+130q^8+100q^7+74q^6+52q^5+34q^4+20q^3+10q^2+4q+1$$
We get $P_2$ from $P_1$ by taking 
$\{\bt,\bv,\bw,\bw^{-1}\bv^{-1}\bu\bv\bw\}$ as
generators.

\subsubsection*{$P_3$} The presentation $P_3$ is
$$\left\langle \bt,\bu,\bv,\bw\mid
  \begin{aligned}
  &\bt\bw=\bw\bt, \bu\bw\bu=\bw\bu\bw, \bu\bv\bu=\bv\bu\bv, \bv\bw\bv=\bw\bv\bw,
   \bt\bu\bt=\bu\bt\bu,\\
  &\bt\bv\bt=\bv\bt\bv, \bu\bv\bw\bu=\bw\bu\bv\bw,
   (\bt\bu\bv)^2=(\bv\bt\bu)^2\\
  \end{aligned}
\right\rangle$$
It appears 72 times in a Hurwitz orbit. Its Poincar\'e polynomial is
$$34q^{10}+88q^9+122q^8+132q^7+111q^6+75q^5+45q^4+25q^3+11q^2+4q+1$$
We get $P_3$ from $P_1$ by taking 
$\{\bt,\bv,\bv^{-1}\bu\bv,\bw\}$ as generators.

\subsubsection*{$P_4$} The presentation $P_4$ is
$$\left\langle \bt,\bu,\bv,\bw\mid
  \begin{aligned}
  &\bt\bv\bt=\bv\bt\bv, \bu\bw\bu=\bw\bu\bw, \bt\bw\bt=\bw\bt\bw, 
    \bt\bw\bv\bt=\bv\bt\bw\bv, \bw\bv\bu\bw=\bu\bw\bv\bu, \\
  &\bt\bu\bt=\bu\bt\bu, \bw\bv\bw=\bv\bw\bv, \bu\bv\bu=\bv\bu\bv, 
   (\bt\bu\bw)^2=(\bw\bt\bu)^2, (\bt\bu\bv)^2=(\bu\bv\bt)^2\\
  \end{aligned}
\right\rangle$$
It appears 36 times in a Hurwitz orbit. Its Poincar\'e polynomial is
$$6q^{10}+40q^9+98q^8+148q^7+149q^6+102q^5+58q^4+30q^3+12q^2+4q+1$$
We get $P_4$ from $P_1$ by taking 
$\{\bt,\bu,\bv,\bv\bw\bv^{-1}\}$ as generators.

\subsubsection*{$P_5$} The presentation $P_5$ is
$$\left\langle \bt,\bu,\bv,\bw\mid
  \begin{aligned}
  &\bw\bu=\bu\bw, \bt\bv\bt=\bv\bt\bv, \bv\bu\bv=\bu\bv\bu,
   \bt\bu\bt=\bu\bt\bu, \bt\bw\bt=\bw\bt\bw,\\
  &\bw\bv\bw=\bv\bw\bv, \bt\bw\bv\bu\bt\bw=\bu\bt\bw\bv\bu\bt,
   (\bv\bu\bt)^2=(\bu\bt\bv)^2,(\bw\bv\bt)^2=(\bv\bt\bw)^2\\
  \end{aligned}
\right\rangle$$
It appears 18 times in a Hurwitz orbit. Its Poincar\'e polynomial is
$$q^{10}+28q^9+97q^8+163q^7+162q^6+104q^5+52q^4+25q^3+11q^2+4q+1$$
We get $P_5$ from $P_1$ by taking 
$\{\bt,\bv,\bu,\bu\bv\bw\bv^{-1}\bu^{-1} \}$ as generators.

In each case we obtain a presentation of $B(G_{33})$ by adding one generator
$\bs$ and the relations
$\bs\bt\bs=\bt\bs\bt, \bs\bu=\bu\bs,\bs\bv=\bv\bs,\bs\bw=\bw\bs$.
We then obtain a presentation of $B(G_{34})$ by adding one generator $\bx$
and the relations $\bw\bx\bw=\bx\bw\bx, \bx\bs=\bs\bx,\bx\bt=\bt\bx,
\bx\bu=\bu\bx, \bx\bv=\bv\bx$, except for the representation corresponding
to $P_2$  where the relations we should add are then
$\bw\bx=\bx\bw, \bs\bx\bs=\bx\bs\bx,\bx\bt=\bt\bx,
\bx\bu=\bu\bx, \bx\bv=\bv\bx$ (this presentation can be obtained from the one
corresponding to $P_1$ by taking $\bs\bt\bu\bt^{-1}\bs^{-1},\bs,\bt,
\bv,\bw,\bx$ as generators).

\section{Representations from $W$-graphs}\label{sec:Wgraphs}
The notion of a $W$-graph for a representation of a Weyl group originates from
Kazhdan-Lusztig theory. Here, we propose a generalization of this concept
to the case of complex reflection groups. We then deal with computational
issues connected with this.
\subsection{$W$-graphs}
Let $W\le\GL(V)$ be a complex reflection group on $V$. We assume that $W$
is {\em well-generated}, that is, $W$ can be generated by $n:=\dim(V)$
reflections $s_1,\ldots,s_n$. We set $I=\{1,2,\ldots,n\}$ and we
let $d_j$ denote the order of the reflection
$s_j$, $j\in I$, and $d:=\max\{d_j\}$.

The following generalizes the concept of a $W$-graph for a representation of
a finite Coxeter group, see \cite{Gy}. Let $R:W\rightarrow \GL_r(\CC)$ be an
irreducible representation of $W$. A {\em pre-$W$-graph}
$\Gamma$ for $R$ is a sequence $(\gamma_1,\ldots,\gamma_r)$ of $r$ maps
$\gamma_k:I\rightarrow\{0,\ldots,d-1\}$ satisfying $\gamma_k(j)\le d_j-1$ for
all $1\le k\le r$, $j\in I$. The maps $\gamma_k$ are also called the
{\em nodes} of $\Gamma$.

We now define the concept of an {\em admissible} pre-$W$-graph. If $W$ is
cyclic, then $n=1$ and $I=\{1\}$. Any irreducible representation is 1-dimensional,
so $r=1$, and the generating reflection $s_1$ acts by
$R(s_1)=\exp(2\pi im/d_1)$ in $R$ for some $1\le m\le d_1$. Then only the
map $\gamma_1$ with $\gamma_1(1)=m$ is admissible.
Now assume inductively that an admissible pre-$W$-graph has been chosen for
each irreducible representation of each proper parabolic subgroup
$W_J=\langle s_j\mid j\in J\rangle$, where $J\subset I$. The pre-$W$-graph
$\Gamma$ is then called admissible (with respect to the chosen admissible
pre-$W$-graphs of the parabolics), if for each parabolic subgroup $W_J <W$
the restriction of $\Gamma$ to $W_J$ is the union of the pre-$W$-graphs of
the restriction of $R$ to $W_J$. (Here, restriction to a parabolic subgroup
$W_J$, $J\subset I$, is obtained by restricting all $\gamma_j$ to $J$.)

Let $\cH(W,\bu)$ denote the generic cyclotomic Hecke algebra associated to
$W$, where $\bu=(u_{j,m}\mid 1\le j\le n, 0\le m\le d_j-1)$ with 
$u_{j,0},\ldots, u_{j,d_j-1}$ corresponding to $s_j$ as above. 
For each $j$ we choose a total ordering
on the variables $u_{j,0},\ldots,u_{j,d_j-1}$; for example
$u_{j,0}>u_{j,1}>\ldots>u_{j,d_j-1}$. We write $T_j$ for the image in
$\cH(W,\bu)$ of a braid reflection mapping to $s_j$, $1\le j\le n$. Given an
admissible pre-$W$-graph $\Gamma$ for a representation $R$ of $W$, we
associate a {\em pre-representation} of $\cH(W,\bu)$ as follows.
For each $j\in I$, let $T_j$ be an $r\times r$-matrix with diagonal
entries $T_j[k,k]=u_{j,\gamma_k(j)}$. The off-diagonal entry $T_j[k,l]$ is
zero unless $T_j[k,k]<T_j[l,l]$ in the chosen ordering on $\bu$.
The remaining entries of the $T_j$ are independent indeterminates, except
that $T_j[k,l]=T_m[k,l]$ if $s_j,s_m$ are conjugate in $W$ and both
$T_j[k,k]=T_m[k,k]$ and $T_j[l,l]=T_m[l,l]$, for $k\ne l$.

Any specialization of these matrices which define a representation of
$\cH(W,\bu)$ which is a deformation of a conjugate of the given representation
$R$ of $W$ is called a {\em $W$-graph} for $R$. Note that, since $W$ is
admissible, the characteristic polynomials of the $T_j$ are by construction
already as they should be if the $T_j$ did define a representation of $\cH(W)$
specializing to $R$. \par
If all reflections of $W$ are conjugate and of order~2, such a representation
can be encoded in an actual labelled and directed graph as follows: the nodes
of the graph are in bijection with the set $\{1,\ldots,r\}$, labelled by
$\gamma_k(1)$ (note that here $\gamma_k$ is already uniquely determined by
$\gamma_k(1)$). There is a directed edge from $j$ to $k$, labelled by $T_m[j,k]$,
if $T_m[j,k]\ne0$ for some $m$ with $\gamma_j(m)=1$ and $\gamma_k(m)>1$.
Note that the value of $T_m[j,k]$ does not depend on the choice of $m$, by
our convention on pre-representations. Clearly the representation can be
recovered from this graph. Note also that there are just two possible total
orderings of the two variables $u_0,u_1$ in this case, and changing the ordering
amounts to transposing the representing matrices.

For  instance,  here  is  the  graph  for  the  representation of $\cH(W,\bu)$
specializing    to   the   reflection    representation,   where   $W=G_{24}$,
$\bu=\{u_{1,0},u_{1,1}\}=\{x,y\}$.   We  consider  the   matrices  for  the  3
generators for presentation $P_1$ in Section~\ref{subsec:G24}. We have $d=2$,
$r=3$.
\begin{equation}\label{phi31}
\xymatrixcolsep{5pc}
\xymatrixrowsep{5pc}
\xymatrix{
12\ar@<3pt>[rr]^{x(-1-\sqrt{-7})/2}\ar@<3pt>[dr]^{-y}&&13
\ar@<3pt>[ll]^{y(1-\sqrt{-7})/2}\ar@<3pt>[dl]^{-y}\\
&\ar@<3pt>[ul]^x\ar@<3pt>[ur]^x 23&\\
}
\end{equation}

\begin{rem}
Assume for a moment that $W$ is a real reflection group, i.e., a finite Coxeter
group. In this case Gyoja has shown \cite{Gy} that any irreducible
representation has a model which comes from a pre-representation of a
$W$-graph of $W$. 
\end{rem}

In general, for arbitrary complex reflection groups, admissible pre-$W$-graphs
need not always exist. But in cases where pre-$W$-graphs do exist,
examples show that
\begin{itemize}
 \item[(a)] often, there exist corresponding representations, 
 \item[(b)] these representation tend to be very sparse,
 \item[(c)] often the entries can be chosen to be Laurent-polynomials
   in the parameters,
 \item[(d)] all $n$ matrices are equally sparse.
\end{itemize}

\subsection{Existence of pre-$W$-graphs and $W$-graphs}
We say that a representation of $W$ has an admissible pre-$W$-graph if there
exists one for {\em some} presentation of the braid group and for {\em some}
ordering of the variables.
It is pretty straightforward to write a program which enumerates all admissible
pre-$W$-graphs for $W$, given those of the proper parabolic subgroups. In order
for this inductive process to work, we have to consider also some rank 2 groups.

It is much more difficult in general to find, given a pre-representation for a
pre-$W$-graph of $W$, specializations of the entries so that it
actually defines a representation of $\cH(W)$ (which specializes to the
representation $R$ of $W$ we started with).

Let us give some indications on how the $W$-graphs presented below were
constructed. For small representations (of dimension
at most~4), this is straightforward, by solving the non-linear system of
equations obtained by requiring that the given pre-representation satisfies
the relations of $\cH:=\cH(W,\bu)$. For larger dimensions, this system becomes
too
large: if the representation $R$ has dimension~$r$ and $\cH$ has $n$
generators, then the matrices of the generators for the pre-representation
involve at least roughly $nr^2/4$ unknowns. A braid relation of length $m$
in the generators produces algebraic equations of degree~$m$ between these
unknowns. Furthermore, the coefficients in the equations involve all the
parameters $\bu$ of the Hecke algebra. A simple minded application of the
Buchberger algorithm to such a system of equations is bound to fail. \par

Therefore, we had to use several tricks. A look at the final representations
shows that they are very sparse, containing many more zero-entries than
required by the definition of pre-representation. Knowing the positions
of these zeros in advance would allow to solve the system of equations
easily. Thus, in a first step, we tried to conjugate the given representation
of $W$ into the form of a $W$-graph, with as many zeros as possible. Note that
the conditions on the entries of such a conjugating matrix are linear, hence
this system is easy to solve. In general, there will not be a unique
solution, but we chose a solution with as many zeros as possible. Then we
looked for a $W$-graph representation of $\cH$ with zero entries in the same
positions. \par
Alternatively, we started from a representation of $\cH$ obtained by the
methods of Section~\ref{sec:hensel}, for example, and tried to conjugate this
to a $W$-graph representation. \par
For dimensions larger than~10, say, even the determination of such a
conjugating matrix becomes too difficult. Here, one successful approach used
information from maximal parabolic subgroups. Let $W_J$ be a maximal parabolic
subgroup of $W$, and assume that the restriction of $R$ to $W_J$ splits as
$$R|_{W_J}=R_1\oplus\ldots\oplus R_t$$
into a sum of irreducible representations $R_i$ of $W_J$. By induction, we may
assume that $W$-graphs of the $R_i$ for $\cH(W_J)$ are already known. In order
to use this information we made the additional assumption that the block
diagonal part of the $T_j$, with $j\in J$, agrees with the $W$-graphs of the
$R_i$. (This doesn't follow from our axioms on pre-$W$-graphs, and it is in
fact not always satisfied.) This 'Ansatz' again reduces the number of unknowns
considerably. Clearly, the larger the dimensions of the $R_i$ are, that is, the
fewer constituents occur, the more information we obtain.

\subsubsection{Some rank-2 groups}
We describe the situation in some detail for the case of the smallest
well-generated exceptional group $G_4$. There is for each representation
exactly one $W$-graph, as given in Table~\ref{tab:G4}. The labelling of
characters is as in \cite{MaG}, for example.

\begin{table}[htbp]
\caption{$W$-graphs for $G_4$}  \label{tab:G4}
$\begin{array}{|c|c||c|c|}
\hline
 \mbox{character}& \mbox{$W$-graph}& \mbox{character}& \mbox{$W$-graph}\\
\hline
\phi_{1,0}&(12..)& \phi_{2,1}&(.12.,12..)\\
\phi_{1,4}&(.12.)& \phi_{2,3}&(..12,12..)\\
\phi_{1,8}&(..12)& \phi_{2,5}&(..12,.12.)\\
\phi_{3,2}&(.12.,1..2,2..1)& & \\
\hline
\end{array}$
\end{table}

Each 2-dimensional representation admits one further pre-$W$-graph, for
instance $\phi_{2,5}$ admits $(.1.2,.2.1)$; however the only $W$-graph
corresponding to it is a non-irreducible representation (which has same
restriction to parabolic subgroups). For the 3-dimensional representation,
there are 5 more pre-$W$-graphs:
$$(..12,.12.,12..),\quad (..12,1.2.,2.1.),\quad (.1.2,.2.1,12..),\quad
(.1.2,1.2.,2..1),\quad (.2.1,1..2,2.1.)$$
The first 3 give rise to non-irreducible representations, and the last two
do not give rise to any representation.

Similarly, for the group $G_{3,1,2}$, each representation admits one
pre-$W$-graph which is a $W$-graph, except for the 2-dimensional representations
which also admit another pre-$W$-graph giving rise to a non-irreducible 
representation. The same situation holds for the Coxeter groups $A_2$,
$B_2$ and $I_2(5)$.

\subsubsection{Pre-$W$-graphs for $G_{25}$}
The inductive approach now gives the following:

\begin{prop}
 For $G_{25}$, each irreducible representation admits a single pre-$W$-graph
 whose restriction to each parabolic subgroup $G_4$ is an actual $W$-graph.
 For each pre-$W$-graph, there exists a corresponding $W$-graph.
\end{prop}

Table~\ref{tab:G25} contains these $W$-graphs for one representation in each
orbit under Galois automorphisms on the parameters.

\begin{table}[htbp]
\caption{$W$-graphs for $G_{25}$}  \label{tab:G25}
$\begin{array}{|c|c|}
\hline
 \mbox{Character}&\mbox{$W$-graph}\\
\hline
\phi_{1,0}& (123..)\\
\phi_{2,3}& (13.2.,2.13.)\\
\phi_{3,6}& (.123.,13..2,2..13)\\
\phi_{3,1}& (12.3.,13.2.,23.1.)\\
\phi_{6,2}& (1.23.,12..3,13..2,2.13.,23..1,3.12.)\\
\phi_{6,4}''& (.123.,1.23.,13..2,2.1.3,2.3.1,3.12.)\\
\phi_{8,3}& (1.23.,12..3,13..2,13..2,2.1.3,2.3.1,23..1,3.12.)\\
\phi_{9,5}& (.123.,1.2.3,1.3.2,13..2,2..13,2.1.3,2.3.1,3.1.2,3.2.1)\\
\hline
\end{array}$
\end{table}

\subsubsection{Pre-$W$-graphs for $G_{26}$}
\begin{prop}
 For $G_{26}$, all but two 6-dimensional irreducible representations admit at
 least one pre-$W$-graph whose restrictions to both parabolic subgroups of
 type $G_4$ and $G_{3,1,2}$ are actually $W$-graphs. 
 For each pre-$W$-graph, there exists a corresponding $W$-graph.
\end{prop}

Table~\ref{tab:G26} contains the unique pre-$W$-graph for the given
representatives of the Galois orbits.
 
 \begin{table}[htbp]
\caption{$W$-graphs for $G_{26}$}  \label{tab:G26}
$\begin{array}{|c|c|}
\hline
 \mbox{Character}&\mbox{$W$-graph}\\
\hline
\phi_{1,0}& (123..)\\
\phi_{2,3}& (12.3.,13.2.)\\
\phi_{3,1}& (12.3.,13.2.,23.1.)\\
\phi_{3,6}& (1.23.,12..3,13..2)\\
\phi_{6,2}& (1.23.,12..3,13..2,13.2.,2.13.,23.1.)\\
\phi_{8,3}& (1.23.,12..3,13..2,13.2.,2.1.3,2.13.,23.1.,3.1.2)\\
\phi_{9,5}& (1.2.3,1.23.,1.3.2,12..3,13..2,13.2.,2.1.3,2.13.,3.1.2)\\
\hline
\end{array}$
\end{table}

\subsubsection{Pre-$W$-graphs for $G_{24}$}
For $G_{24}$, the situation depends on the presentation of the braid group
$B(G_{24})$ considered, see Section~\ref{subsec:G24}.

\begin{prop}
 For $G_{24}$, for each of the presentations $P_1$ to $P_3$, each
 representation admits at most one pre-$W$-graph whose restrictions to
 parabolic subgroups of type $A_2$ and $B_2$ are $W$-graphs. But for $P_2$ and
 $P_3$ (the same) eight of the twelve representations admit such a graph, while
 for $P_1$ two more representations admit such pre-$W$-graphs.
 For each pre-$W$-graph for $P_1$, there exists a corresponding $W$-graph.
\end{prop}

Table~\ref{tab:G24} contains the pre-$W$-graphs for the given representatives
of the Galois orbits and presentation $P_1$.
 
\begin{table}[htbp]
\caption{$W$-graphs for $G_{24}$, presentation $P_1$}  \label{tab:G24}
$\begin{array}{|c|c|}
\hline
 \mbox{Character}&\mbox{$W$-graph}\\
\hline
\phi_{1,0}&(132)\\
\phi_{3,1}&(13,12,32)\\
\phi_{6,2}&(13,13,12,12,32,32)\\
\phi_{7,6}&(1,13,12,3,3,2,2)\\
\hline
\end{array}$
\end{table}

The representations $\phi_{8,4}$ and $\phi_{8,5}$ do not admit any
pre-$W$-graph.

A $W$-graph for $\phi_{3,1}$ has been given in \ref{phi31}. With the same
conventions, here is a $W$-graph for $\phi_{6,2}$:

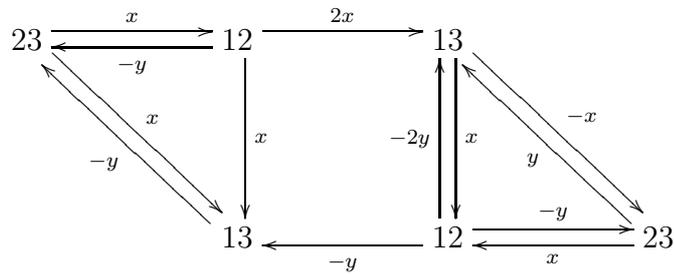
\begin{equation*}
\xymatrixcolsep{5pc}
\xymatrixrowsep{5pc}
\xymatrix{
23\ar@<3pt>[r]^x\ar@<3pt>[dr]^x&
12\ar@<3pt>[l]^{-y}\ar@<3pt>[d]^x\ar@<3pt>[r]^{2x}&
13\ar@<3pt>[d]^x\ar@<3pt>[dr]^{-x}&\\
&13\ar@<3pt>[ul]^{-y}&
12\ar@<3pt>[l]^{-y}\ar@<3pt>[u]^{-2y}\ar@<3pt>[r]^{-y}&
23\ar@<3pt>[l]^x\ar@<3pt>[ul]^y\\
}
\end{equation*}

And here is a $W$-graph for $\phi_{7,6}$:
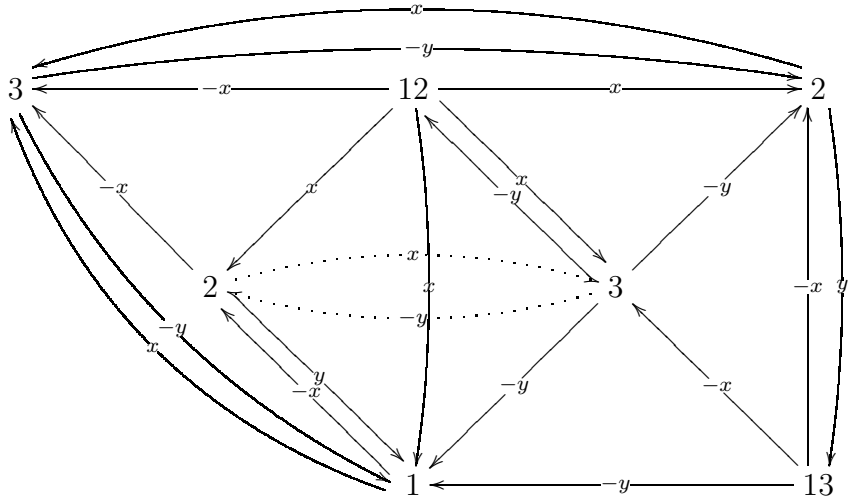
\begin{equation*}
\xymatrixcolsep{5pc}
\xymatrixrowsep{5pc}
\xymatrix{3\ar@<-3pt>@/_1.6pc/[ddrr]|{-y}\ar@<3pt>@/^1pc/[rrrr]|{\,-y}&&
12\ar[ll]|{\,-x}\ar[rr]|x\ar[dl]|x\ar@<4pt>[dr]|x\ar@/^/[dd]|x&&
2\ar@<-6pt>@/_2pc/[llll]|x\ar@<3pt>@/^/[dd]|y\\
&2\ar[ul]|{-x}\ar@<4pt>[dr]|y
\ar@/^1pc/@{.>}[rr]|x&&
3\ar@<4pt>[ul]|{-y}\ar[ur]|{-y}\ar[dl]|{-y}
\ar@/^1pc/@{.>}[ll]|{-y}&\\
&&1\ar@<6pt>@/^2.2pc/[uull]|x\ar@<3pt>[ul]|{-x}&&
13\ar[ul]|{-x}\ar@<4pt>[uu]|{-x}\ar[ll]|{-y}\\
}
\end{equation*}

Together with the matrices given above for $\phi_{8,5}$, this completes the
description of the representations of the Hecke algebra of $G_{24}$.

\subsubsection{Pre-$W$-graphs for $G_{27}$}

\begin{prop}
 For the presentations $P_1$ to $P_5$ of $G_{27}$, each representation admits
 at most one pre-$W$-graph whose restrictions to parabolic subgroups of type
 $A_2$, $B_2$ and $I_2(5)$ are $W$-graphs. For $P_1$ 26 out of 34
 representations admit such a graph, while for $P_2$ (resp. $P_3,P_4,P_5$) just
 16, (resp. 20, 14, 16) admit such a graph. Moreover, any representation which
 admits such a graph for any of $P_2$--$P_5$ admits one for $P_1$.
 For each pre-$W$-graph for $P_1$, there exists a corresponding $W$-graph.
\end{prop}

 Table~\ref{tab:G27} contains the pre-$W$-graphs for $P_1$ for
 representatives of the Galois orbits.
The $8$ and $15$-dimensional representations do not admit any pre-$W$-graph.

\begin{table}[htbp]
\caption{$W$-graphs for $G_{27}$}  \label{tab:G27}
$\begin{array}{|c|c|}
\hline
 \mbox{Character}&\mbox{$W$-graph}\\
\hline
\phi_{1,0}&(132)\\
\phi_{3,1}&(12,13,23)\\
\phi_{5,6}''&(12,13,13,2,23)\\
\phi_{5,6}'&(12,12,13,23,3.12)\\
\phi_{6,2}&(13,13,12,12,23,23)\\
\phi_{9,6}&(1,12,12,13,13,2,23,23,3)\\
\phi_{10,3}&(12,12,12,13,13,13,2,23,23,3)\\
\hline
\end{array}$
\end{table}
  
Here is a $W$-graph for $\phi_{3,1}$, where $c=1+\zeta_3^2(1-\sqrt 5)/2$:
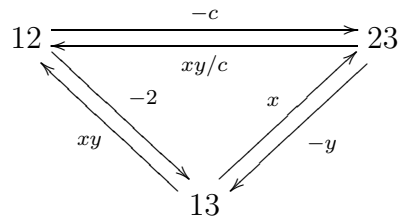
\begin{equation*}
\xymatrixcolsep{4pc}
\xymatrixrowsep{4pc}
\xymatrix{
12\ar@<3pt>[rr]^{-c}\ar@<3pt>[dr]^{-2}&&23
\ar@<3pt>[ll]^{xy/c}\ar@<3pt>[dl]^{-y}\\
&\ar@<3pt>[ul]^{xy}\ar@<3pt>[ur]^x 13&\\
}
\end{equation*}

Here is a $W$-graph for $\phi_{5,6}'$,
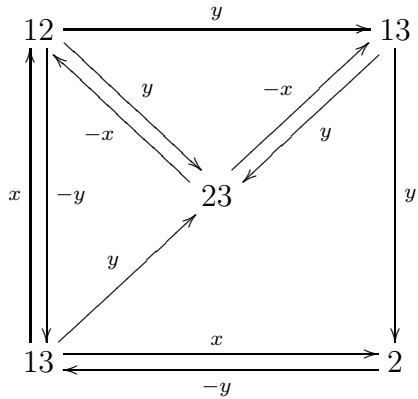
\begin{equation*}
\xymatrixcolsep{4pc}
\xymatrixrowsep{4pc}
\xymatrix{
12\ar[rr]^y\ar@<3pt>[dr]^y\ar@<3pt>[dd]^{-y}&&13\ar@<3pt>[dl]^y\ar[dd]^y\\
&23\ar@<3pt>[ul]^{-x}\ar@<3pt>[ur]^{-x}\\
13\ar@<3pt>[uu]^x\ar[ur]^y\ar@<3pt>[rr]^x&&2\ar@<3pt>[ll]^{-y}\\
}
\end{equation*}
and here is one for $\phi_{6,2}$:
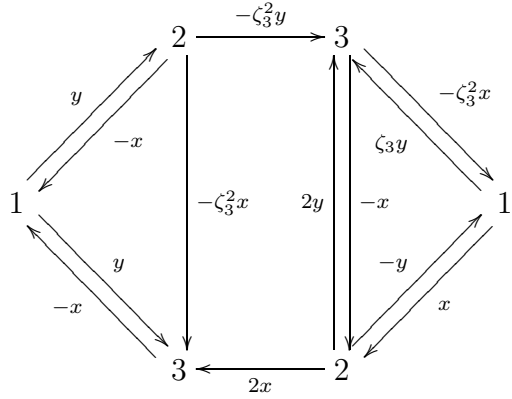
\begin{equation*}
\xymatrixcolsep{4pc}
\xymatrixrowsep{4pc}
\xymatrix{
&2\ar[r]^{-\zeta_3^2y}\ar@<3pt>[dl]^{-x}\ar@<3pt>[dd]^{-\zeta_3^2x}&
3\ar@<3pt>[dr]^{-\zeta_3^2x}\ar@<3pt>[dd]^{-x}&\\
1\ar@<3pt>[ur]^y\ar@<3pt>[dr]^y&&&
1\ar@<3pt>[ul]^{\zeta_3y}\ar@<3pt>[dl]^x\\
&3\ar@<3pt>[ul]^{-x}&2\ar[l]^{2x}\ar@<3pt>[uu]^{2y}\ar@<3pt>[ur]^{-y}\\
}
\end{equation*}
Interchanging the generators $\bu$,$\bt$ in the presentation $P_1$
obviously defines an antiautomorphism of $B(W)$. It can be checked that
composition of this antiautomorphism with transposition interchanges the
representations $\phi_{5,6}'$ and $\phi_{5,6}''$, so we need not give a
$W$-graph for the latter.

We now give a $W$-graph for $\phi_{9,6}$ as the union of the following pieces.
The nodes are $1$, $2$, $3$, $12$, $\fbox{12}$, $13$, $\fbox{13}$,
$23$, $\fbox{23}$ (the last three occur twice and we box one of the occurrences
to distinguish it from the other). Here we set $u=\sqrt[3]x$ and $v=\sqrt[3]y$.
$$
3\xleftarrow{u^3}\fbox{12}\xrightarrow{v^2-uv}\fbox{13}\xrightarrow{-u^3v}12
\xrightarrow{(u-v)(u^2+v^2)}1\xrightarrow{-u^3v^3}\fbox{23}\xrightarrow{1}2
\xleftarrow{(u-v)^2}23
$$

$$
\fbox{13}\xleftarrow{-v^2}2\xleftarrow{v^2(u-v)}
\fbox{12}\xrightarrow{-uv^2}13\xrightarrow{v^2(u-v)}1
\xleftarrow{1}\fbox{23}\xrightarrow{-1}3\xleftarrow{v(2u-v)}23
\xleftarrow{-u^2v^2}12
$$

$$
13\xrightarrow{-u^2v^2}23\xrightarrow{uv}12
\xrightarrow{v^2}\fbox{13}\xrightarrow{u^2v(v-u)}3\xrightarrow{v^2(u-v)}2
$$

$$
2\xleftarrow{u^3v}\fbox{13}\xleftarrow{u-v}23\xrightarrow{uv}13
\xrightarrow{u^2v}\fbox{12}\xrightarrow{u(uv-u^2-v^2)}1
$$

$$
2\xleftarrow{u^2(v-u)}
13\xrightarrow{u^2(u-v)}3\xrightarrow{-v^3}\fbox{12}
$$
Here is, following the same conventions, a $W$-graph for $\phi_{10,3}$.
The nodes are $13$, $12$, $\fbox{13}$, $\fbox{12}$, $\fbox{\fbox{13}}$,
$\fbox{\fbox{12}}$, $3$, $2$, $23$, $\fbox{23}$:

$$
13\xrightarrow{y} 3\xrightarrow{-1} \fbox{12}\xrightarrow{-y}
\fbox{\fbox{13}}\xrightarrow{xy+x^2+y^2} 12\xrightarrow{x} 
\fbox{23}\xrightarrow{-1} \fbox{13}\xrightarrow{x} 
\fbox{\fbox{12}}\xrightarrow{x-y} 2\xleftarrow{-x+2y} 23
$$

$$
13\xrightarrow{-1} 23\xrightarrow{-xy-y^2} 12\xleftarrow{2x^2+y^2}
\fbox{13}\xrightarrow{x} \fbox{12}\xrightarrow{xy} 3\xrightarrow{1}
2\xrightarrow{-y} \fbox{\fbox{13}}\xleftarrow{3y} 
\fbox{\fbox{12}}\xrightarrow{xy+x^2} \fbox{23}
$$

$$
2\xleftarrow{x} \fbox{\fbox{13}}\xrightarrow{x^2} 3\xleftarrow{2x^2}
\fbox{13}\xleftarrow{-2y} \fbox{\fbox{12}}\xrightarrow{-x} 23\xrightarrow{-y}
\fbox{12}\xrightarrow{xy+x^2+y^2} 13\xrightarrow{x} \fbox{23} 
\xrightarrow{-y} 12
$$

$$
\fbox{13}\xrightarrow{2x} 2\xrightarrow{-xy} 3\xleftarrow{xy}
\fbox{\fbox{12}}\xrightarrow{xy+x^2} 13\xleftarrow{xy} 23\xrightarrow{2x}
\fbox{\fbox{13}}\xrightarrow{x} \fbox{12}
$$

$$
\fbox{12}\xrightarrow{x+y} 2\xleftarrow{-1} \fbox{23}\xleftarrow{xy}
\fbox{13}\xleftarrow{-x} 23\xrightarrow{y} \fbox{\fbox{12}}
$$

$$
23\xrightarrow{y^2} 3\qquad
\fbox{23}\xrightarrow{2} \fbox{\fbox{13}}
$$

We have found a model for $\phi_{15,5}$ by Hensel lifting, but at the current
time we could not find a model with coefficients in $A$, only in $K$.

We now give some information on higher dimensional primitive groups:

\subsubsection{Pre-$W$-graphs for $G(4,4,3)$ and $G_{29}$}
For the presentation of $G(4,4,3)$ corresponding to $P_1$ the 6-dimensional
representation does not admit a pre-$W$-graph while for the one corresponding
to $P_2$ it is the 2-dimensional representation which does not admit one.

\begin{prop}
 For the presentation corresponding to $P_1$ of $G_{29}$, 15 representations 
 admit a pre-$W$-graph, while for that corresponding to $P_2$, 27 (out of 37)
 admit one.
 The two representations $\phi_{5,8}$ and $\phi_{5,16}$ admit a pre-$W$-graph
 for $P_1$ and not for $P_2$. \par
 All together all representations of $W$ admit a pre-$W$ graph except for two of
 the 4 of dimension 15, and for those of dimension 20.
 We have found actual $W$-graphs for all the pre-$W$-graphs excepted the last
 three of table~\ref{tab:G29}.
\end{prop}

Table~\ref{tab:G29} contains pre-$W$-graphs for representatives of Galois
orbits of representations of the Hecke algebra. The first 7 graphs in the
table correspond to $P_1$ and the rest to $P_2$. To condense the table,
repeated nodes are represented once, the multiplicity being given by an
exponent.

\begin{table}[htbp]
\caption{pre-$W$-graphs for $G_{29}$}  \label{tab:G29}
$\begin{array}{|c|c|}
\hline
 \mbox{Character}&\mbox{$W$-graph}\\
\hline
\phi_{1,0}&(1234)\\
\phi_{4,4}&(123,124,134,234)\\
\phi_{4,1}&(123,124,134,234)\\
\phi_{5,8}&(123,134,14,23,24)\\
\phi_{6,12}&(12,13,14,23,24,34)\\
\phi_{6,10}'''&(12,13,14,23,24,34)\\
\phi_{10,2}&(123^2,124^2,134^3,23,234,24)\\
\hline
\phi_{6,10}'&(13,134,142,32,4,2)\\
\phi_{10,6}&(13,134,132,14,142,12,34,342,32,42)\\
\phi_{15,4}'&(13,134^3,132^2,14,142^2,34,32^2,42^2,2)\\
\phi_{16,3}&(13,134^3,132^2,14,142^2,12,34,342, 32^2,42^2)\\
\phi_{24,6}&(13^3,134^2,132,14^3,142,12^2,3,34^2,32^3,4,42^3,2^2)\\
\phi_{24,7}&(13^3,134^2,132,14^3,142,12^2,3,34^2,32^3,4,42^3,2^2)\\
\phi_{30,8}&(13^4,134^2,132,14^4,142,12^3,3,34^3,32^4,4,42^4,2^2)\\
\hline
\end{array}$
\end{table}

\subsubsection{Pre-$W$-graphs for $G_{32}$}
For $G_{32}$, 57 of the 102 irreducible representations admit a pre-$W$-graph.  
If one includes the Galois-conjugates of these representations, one gets all
representations but 12: the missing ones are 3 of the 60-dimensional ones,
the 64-dimensional and the 81-dimensional ones.

\subsubsection{Pre-$W$-graphs for $G_{33}$}
For the presentation $P_1$ of \ref{subsec:G33} we find that 14 representations
admit  a pre-$W$-graph, and for presentation $P_2$  we find that 14 more admit
one,  for  a  total  of  28 out of 40.  The set of representations which admit a
pre-$W$-graph is stable under Galois action, so we do not get new ones.
For the other presentations $P_3$ to $P_5$ the representations which admit a
pre-$W$-graph are a subset of the 14 which have one for $P_1$.

\subsubsection{Pre-$W$-graphs for $G_{34}$}
For the presentation corresponding to $P_1$ of \ref{subsec:G33} we find that
18 representations admit  a pre-$W$-graph, and for the presentation
corresponding to $P_2$  we find that 13 more admit one,  for  a  total  of  
31 out of 169. The set of representations which admit a
pre-$W$-graph is stable under Galois action, so this does not give new ones.
For the presentations corresponding to $P_3, P_4, P_5$ the representations 
which admit a pre-$W$-graph are a subset of the 18 which have one for $P_1$.

\section{Checking the conjectures of section \S2}

We now use the representations obtained above for 2 and 3-dimensional
exceptional groups in order to verify some
of the conjectures on the structure of cyclotomic Hecke algebras stated in
Section~\ref{sec:conj} for some of the primitive complex reflection groups.

\subsection{Computational difficulties}

The main problem to  carry out the computations implied by e.g.
Proposition~\ref{good} is to compute the form $t$ on a large set of images
of elements of $B(W)$ (a set of cardinality $r|W|^2$ where $r$ is the rank of
$W$).

To  compute $t$, we use formula (\ref{schur}), where the Schur elements are
taken  from \cite{MaG}  and $\chi(x)$  is computed  using the matrices for the
representation of character $\chi$ that we computed in the previous sections.

To minimize the computations, we use a few tricks:
\begin{itemize}
\item  We compute the orbits  of the set of  words we consider under the braid
 relations  and rotations (which give a  conjugate element in the braid group).
 It is sufficient to compute the trace on one element of each orbit.
\item  When all generators of $W$  are of order 2, if  in one of the orbits we
 have   a   word   of   length   $k$   where  there  is  a  repetition  $\ldots
 \bs\bs\ldots$,  using the quadratic defining relation of the Hecke algebra
 $(T_{\bs}-u_{s,0})(T_{\bs}-u_{s,1})=0$  we can  reduce the computation
 to that for one word of length $k-1$ and one word of length $k-2$.
\end{itemize}

For instance, for $G_{24}$ to compute the matrix $\{t(T_{\bw\bw'})\}$
we have to compute the trace on $336^2=112896$ elements; they fall into
$14334$ orbits under rotations and braid relations, and after taking into account 
quadratic relations we still have to handle $327$ elements. For computing
the matrix products corresponding to these words, we look for the occurrence of
common subwords so as to never compute twice the same product, which means that
for each representation we have about $600$ matrix products to effect.
We also take into account the Galois action on representations so we need to compute
the character value only for one representation in each Galois orbit; for $G_{24}$
there are 5 such orbits.

Even  with these simplifications, the matrix  products for algebras which have
many  parameters get very costly, as well  as the final step of evaluating the
right-hand side of~(\ref{schur}), since the gcd of the Schur elements is a large
polynomial.  In quite  a few  cases we  could only  make a heuristic check, by
computing  in the algebra where the parameters are specialized to prime powers
(to  primes taken to the $e$-th power,  so the algebra splits over $\QQ$). The
heuristic  check for belonging to  $A$ become to belong  to $\ZZ$ localized at
the  chosen primes, and to be a unit in $A$ becomes being an integer with only
prime factors the chosen primes.

\subsection{Finding a section such that $t(T_\bw)=0$}
In  order to find a  section $\bW\subset B(W)$ such  that $t(T_\bw)=0$ for any
$\bw\ne  1$, our first idea was, mimicking  the case of finite Coxeter groups,
to  lift elements  of $W$  by lifting  minimal-length expressions  for them 
as positive words in the generators  $s_1,\ldots,s_n$. However, though  this 
{\em almost}
works, it does not always work; we manage with a slight variation on this, as
we shall explain.

Obtaining  all minimal  length expressions  for elements  of $W$ is quite easy
using  standard methods for  enumerating elements of  a group, and is feasible
for   all  exceptional  complex  reflection  groups  but  $G_{34}$.  

The  Tables~\ref{tab:dim2}  and~\ref{tab:dim3}  below  collect  in  the column
``$t(T_\bw)\ne  0$'' the results we  got by computing the  trace on minimal length
elements.  The number in  the column is  the number of  elements $w\in W$ such
that  some minimal word $\bw$  for $w$ has $t(T_\bw)\ne  0$. If this number is
not  $0$, the second number  separated by a {\tt  /} is the number of elements
$w\in W$ such that no minimal word $\bw$ for $w$ has $t(T_\bw)=0$.

When some minimal word for any $w\in W$ has zero trace, we build a section
by choosing arbitrarily such a word for each element.  We now describe 
how to build a section in the other cases:

For $G_{11}$, with the notations of Section~\ref{sec:2-dim},
the lift all minimal lengths expressions of the two elements
$s_3^2(s_2s_1s_3)^2,(s_2s_1s_3)^2s_3s_2$ have non-zero trace; but the
longer lifts $\bs_1\bs_2\bs_1\bs_3\bs_2\bs_1\bs_3^2\bs_1\bs_3$, 
$\bs_1\bs_3\bs_2\bs_1\bs_3^2(\bs_2\bs_1)^2$ have a zero trace.
By making these picks, we can find a section which satisfies
$t(T_{\bw})=t(T_{\bw^{-1}\bpi})=0$.

For $G_{15}$, all minimal length expressions of the two elements
$s_2(s_1s_3)^2s_2,s_2s_1(s_3s_2)^2$ have $t(T_\bw)\ne0$; but the longer lifts
$(\bs_2\bs_2\bs_3)^2 \bs_2, \bs_2(\bs_1\bs_3)^2\bs_1\bs_2^2$ have a zero 
trace. All minimal expressions of $\bs_1\bs_2(\bs_3\bs_2)^2,
\bs_1\bs_2(\bs_2\bs_3)^2, (\bs_3\bs_2)^2\bs_1\bs_2$ have
$t(T_{\bw^{-1}\bpi})\ne 0$. But the longer lifts $\bs_1(\bs_3\bs_2)^3\bs_ 3,
(\bs_3\bs_2)^3 \bs_3\bs_1,
\bs_3(\bs_2\bs_1)^3\bs_3$ work. By making these picks, we can find a section which satisfies
$t(T_{\bw})=t(T_{\bw^{-1}\bpi})=0$ and $t(T_{\bw\bw'})\in A$, but unfortunately
$\det\{t(T_{\bw\bw'})\}_{w,w'}$ is not invertible in $A$ for this choice.

However, there is another way to build a section which leads to a good
section. Since $G_{11}$ and $G_{15}$ have the same hyperplane arrangements,
the braid group $B(G_{15})$ is a subgroup of index 2 of $B(G_{11})$,
generated by $\bs_1,\bs_2,\bs_3^2$. The elements in the above section for 
$G_{11}$ where $\bs_3$ occurs an even number of times form a section
for $G_{15}$ which turns out to be good.

For $G_{24}$ and $G_{27}$ we only consider the presentation $P_1$ as it is the
best  behaved. For  $G_{27}$, all  minimal lengths  expressions of the element
$(sut)^5$  have  non-zero  trace.  The  center  of  $B(W)$ is generated by the
element  $\bz=(\bs\bt\bu)^5$. The ``bad'' element $(\bs\bu\bt)^5$ is a lift of
$z^{-1}$.  The  lift  $\bz^5$  of  $z^{-1}$,  which  is much longer, satisfies
$t(T_{\bz^5})=0$.

\subsection{Checking that the section is good}
To  check
(\ref{tracecond1}) we avoid having to give an expression for $\bw^{-1}\bpi$ in
terms of the generators by using that
$\chi(T_{\bw^{-1}\bpi})=\chi((T_\bw)^{-1})\omega_\chi(T_\bpi)$, where
$\omega_\chi(T_\bpi)$ is easy to compute using e.g. the formula
\cite[1.22]{BMM}.

In the column ``good'' in Tables~\ref{tab:dim2} and~\ref{tab:dim3} below we
have recorded with a '{\tt +}' if we could check that the section build in the
previous subsection is good, and satisfies the assumptions of Proposition
\ref{good}, thus providing an $A$-basis of $\cH(W,\bu)$.

\subsection{Tables}
In Tables~\ref{tab:dim2} and~\ref{tab:dim3} below we collect the computational
results that we got so far, together with results from an unpublished note of
J\"urgen M\"uller \cite{Mue}. Here, we write 'specialized' in the column
'algebra' when we had to do the computation with parameters specialized to
prime powers.
\par
M\"uller used Linton's vector enumerator to construct the regular representation
of some cyclotomic Hecke algebras $\cH(W)$. From that, he is able to verify
Conjecture~\ref{conj}(a) in several cases by exhibiting an $A$-basis, marked
by a '+' in the column 'rank' of our tables. Furthermore, he can construct a
symmetrizing form over $A$ satisfying Conjecture~\ref{conj}(b) in the cases
marked '+' in the column 'form'.

\begin{table}[htbp]
\caption{Hecke algebras for 2-dimensional primitive groups}  \label{tab:dim2}
$\begin{array}{|c|r|r|c|c|c|c||c|c|}
\hline
 W& |W|& |\bu|&\mbox{algebra}&t(T_\bw)\ne 0&t(T_{\bw^{-1}\bpi})\ne 0&
  \mbox{good}& \mbox{rank \cite{Mue}}& \mbox{form \cite{Mue}}\\
\hline
 G_{4}&  24& 3& +&0&0&+& +& +\\
 G_{5}&  72& 6&\mbox{spec.}&0&0&+& +& +\\
 G_{6}&  48& 5&\mbox{spec.}&0&0&+& +& +\\
 G_{7}& 144& 8&\mbox{spec.}&3/0&1/0&+& +& +\\
 G_{8}&  96& 4&\mbox{spec.}&0&0&+& +& +\\
 G_{9}& 192& 6&\mbox{spec.}&0&0&+& +& +\\
G_{10}& 288& 7&\mbox{spec.}&2/0&2/0&+& +& +\\
G_{11}& 576& 9&\mbox{spec.}&22/2&12/0&+& +& ?\\
G_{12}&  48& 2& +&0&0&+& +& +\\
G_{13}&  96& 4&\mbox{spec.}&1/0&0&+& +& +\\
G_{14}& 144& 5&\mbox{spec.}&0&0&+& +& +\\
G_{15}& 288& 7&\mbox{spec.}&11/2&11/3&+& +& ?\\
G_{16}& 600& 5&\mbox{spec.}&11/0&11/0&?& +& ?\\
G_{20}& 360& 3&\mbox{spec.}&2/0&2/0&+& +& ?\\
G_{21}& 720& 5&\mbox{spec.}&6/0&6/0&?& +& ?\\
G_{22}& 240& 2&+&1/0&4/0&+& +& ?\\
\hline
\end{array}$
\end{table}

Neither M\"uller nor we have been able to check any of the cases
$G_{17},G_{18},G_{19}$, for which the number of parameters is at least~7 and the
order of $W$ at least~1200.

\begin{table}[ht]
\caption{Hecke algebras for 3-dimensional primitive groups}  \label{tab:dim3}
$\begin{array}{|l|r|r|c|c|c|c||c|}
\hline
 W& |W|& |\bu|&\mbox{algebra}&t(T_\bw)\ne 0&t(T_{\bw^{-1}\bpi})\ne 0&
  \mbox{good}&\mbox{rank \cite{Mue}}\\
\hline
G_{24},P_1&  336& 2& +&0&0&+& +\\
G_{24},P_2&     & 2& +&3/0&4/0&+& +\\
G_{24},P_3&     & 2& +&0&0&+& +\\
G_{25}&      648& 3&\mbox{spec.}&0&0&+& +\\
G_{26}&     1296& 5&\mbox{spec.}&0&0&?& +\\
G_{27},P_1& 2160& 2& +&1/1&30/6&?& \\
G_{27},P_2&     & 2& +&41/1&97/28&?& +\\
G_{27},P_3&     & 2& +&31/9&44/24&?& +\\
G_{27},P_4&     & 2& +&19/2&42/1&?& \\
\hline
\end{array}$
\end{table}

In his computations, M\"uller only looked at the presentations $P_2,P_3$ for
the group $G_{27}$.


\end{document}